\documentclass{amsart}
\usepackage{amssymb,amsmath,url,graphicx}
\usepackage{palatino} 
\usepackage[all]{xy}
\usepackage{graphicx}
\usepackage{fullpage}
\usepackage{color} 
\usepackage{amsthm}
\usepackage{enumerate} 
\usepackage{amsrefs} 
\usepackage{float}
\usepackage{caption}
\usepackage{subcaption}
\usepackage{cancel}
\usepackage{cases}
\usepackage[dvipsnames]{xcolor}

\usepackage{amscd}
\usepackage{longtable}
\usepackage{graphicx}
\usepackage{comment}
\usepackage{tikz-cd}

\usepackage{mathtools}
\usetikzlibrary{cd}
\usetikzlibrary{decorations.pathmorphing}

\newif\ifdebug                                                      %
\debugfalse

\newif\ifinfootnote
\infootnotefalse

\let\footnoteasusual\footnote
\renewcommand{\footnote}[1]
{\infootnotetrue\footnoteasusual{#1}\infootnotefalse}

\newcommand{\printname}[1]
{\ifmmode{ \smash{ \raisebox{5pt}{\text{\tiny{#1}}} } }
 \else   {\ifinfootnote \smash{\raisebox{0pt}{\tiny{#1}}}
             \else { \marginpar{
                     \smash{ \makebox[0pt]{\raisebox{-12pt}{\tiny{#1}}} }
                               } } \fi} \fi}

\let\labelasusual\label
\newcommand{\labell}[1]
{\ifdebug {\labelasusual{#1}\printname{{#1}}} \else {\labelasusual{#1}} \fi}

\numberwithin{equation}{section}
\newtheorem {theorem}[equation]         {Theorem}
\newtheorem {lemma}[equation]           {Lemma}

\newtheorem {claim*}                    {Claim}

\newtheorem {corollary} [equation]      {Corollary}
\newtheorem {proposition}  [equation]   {Proposition}

\theoremstyle{definition}
\newtheorem{definition}[equation]{Definition}

\theoremstyle{remark}
\newtheorem{remark}[equation]{Remark}
\newtheorem*{remark*}{Remark}
\newtheorem{example}[equation]{Example}

\newtheorem{question}[equation]{Question}

\newcommand{\CC}{\mathbb{C}}

\newcommand{\tr}{\mathrm{Tr}}

\newcommand{\spec}{\mathrm{Spec}}

\newcommand{\hess}{\mathrm{Hess}}
\newcommand{\flag}{\mathrm{Flags}}
\newcommand{\hgt}{\mathrm{ht}}
\newcommand{\init}{\mathrm{in}}
\newcommand{\ass}{\mathrm{Ass}}

\title[Gr\"obner bases for regular nilpotent Hessenberg varieties and toric degenerations]
{Geometric vertex decomposition, Gr\"obner bases, and Frobenius splittings for regular nilpotent Hessenberg varieties}

\author{Sergio Da Silva}
\address{Dept.\ of Mathematics and Economics, Virginia State University, 
1 Hayden Drive, 
Petersburg, Virginia 23806, USA}
\email{sdasilva@vsu.edu, smd322@cornell.edu}

\author{Megumi Harada}
\address{Dept.\ of Mathematics and Statistics, McMaster University, 
1280 Main Street West, 
Hamilton, Ontario L8S 4K1, Canada}
\email{Megumi.Harada@math.mcmaster.ca}

\date{\today}


\begin{document}

\maketitle

\begin{abstract}
We initiate a study of the Gr\"obner geometry of local defining ideals of Hessenberg varieties by studying the special case of regular nilpotent Hessenberg varieties in Lie type A, and focusing on the affine coordinate chart on $\flag(\mathbb{C}^n) \cong GL_n(\mathbb{C})/B$ corresponding to the longest element $w_0$ of the Weyl group $S_n$ of $GL_n(\mathbb{C})$. Our main results are as follows. Let $h$ be an indecomposable Hessenberg function. We prove that the local defining ideal $I_{w_0,h}$ in the $w_0$-chart of the regular nilpotent Hessenberg variety $\hess(\mathsf{N},h)$ associated to $h$ has a Gr\"obner basis with respect to a suitably chosen monomial order. Our Gr\"obner basis consists of a collection $\{f^{w_0}_{k,\ell}\}$ of generators of $I_{w_0,h}$ obtained by Abe, DeDieu, Galetto, and the second author. We also prove that $I_{w_0,h}$ is geometrically vertex decomposable in the sense of Klein and Rajchgot (building on work of Knutson, Miller, and Yong). We give two distinct proofs of the above results. We make this unconventional choice of exposition because our first proof introduces and utilizes a notion of a \textbf{triangular complete intersection} which is of independent interest, while our second proof using liaison theory is more likely to be generalizable to the general $w$-charts for $w \neq w_0$.  Finally, using our Gr\"obner analysis of the $f^{w_0}_{k,\ell}$ above and for $p>0$ any prime, we construct an explicit Frobenius splitting of the $w_0$-chart of $\flag(\mathbb{C}^n)$ which simultaneously compatibly splits all the local defining ideals of $I_{w_0,h}$, as $h$ ranges over the set of indecomposable Hessenberg functions. This last result is a local Hessenberg analogue of a classical result known for $\flag(\mathbb{C}^n)$ and the collection of Schubert and opposite Schubert varieties in $\flag(\mathbb{C}^n)$. 
\end{abstract}

\section{Introduction}
\labell{sec:intro}

Hessenberg varieties are subvarieties of the full flag variety $\flag(\mathbb{C}^n)$, and the investigation of their properties lies in the fruitful intersection of algebraic geometry, representation theory, and combinatorics, among other research areas.\footnote{In this manuscript, we restrict to the case of Lie type $A$, i.e., when the flag variety corresponds to the group $GL_n(\mathbb{C})$ (or $SL_n(\mathbb{C})$). Much can be said about other Lie types, but we do not delve into that here.} First introduced to the algebraic geometry community by De Mari, Procesi, and Shayman \cite{DeMPS}, they have recently garnered attention due in part to their connection to the well-known and unresolved Stanley-Stembridge conjecture in combinatorics (see e.g. \cite{HarPre} for a leisurely account of some of the history). However, there are many other reasons aside from the Stanley-Stembridge conjecture that Hessenberg varieties are of interest; for instance, they arise in the study of quantum cohomology of flag varieties, they are generalizations of the Springer fibers which arise in geometric representation theory, total spaces of families of suitable Hessenberg varieties support interesting integrable systems \cite{AbeCrooks}, and the study of some of their cohomology rings suggests that there is a rich Hessenberg analogue to the theory of Schubert calculus on $\flag(\mathbb{C}^n)$ \cite{HarTym}. 

It is this last point of view of Schubert calculus, or more specifically the geometry of Schubert varieties, which motivates the present paper.  Specifically, in this manuscript we study \textbf{local patches of Hessenberg varieties} - i.e. intersections of these Hessenberg varieties with certain choices of affine Zariski-open subsets of $\flag(\mathbb{C}^n)$. 
To be more specific, the main results of this manuscript concern the \textbf{local Hessenberg patch ideal}, denoted $I_{w_0,h}$, for the special case of a \textbf{regular nilpotent Hessenberg variety} (to be defined in Section~\ref{subsec: reg nilp Hess}), intersected with the affine coordinate chart $w_0U^-B$ of $\flag(\mathbb{C}^n) \cong GL_n(\mathbb{C})/B$ centered at the permutation flag corresponding to the maximal element in $S_n$. (It is possible to consider the intersection with the chart $wU^-B$ for arbitrary permutations $w$, but we mostly restrict to the longest permutation $w_0$ in this manuscript. Details are in Section~\ref{sec:background}.)

The analogous study of local patches of Schubert varieties is a classical topic and a great deal is known about the corresponding (local defining) ideals, from which properties of Schubert varieties can be deduced. For example, on these local patches, Schubert varieties can be degenerated to a square-free monomial ideal which is associated to a subword complex, which gives another proof that Schubert varieties are Cohen-Macaulay \cite{KM}. We view the present manuscript as a first step in this direction in the context of Hessenberg varieties, in the sense that we initiate a study of the Gr\"obner geometry associated to local defining ideals of local patches of Hessenberg varieties. This being said, we emphasize that we are not the first to consider these local patches; Insko and Yong studied the special case of the Peterson variety using such local patches in \cite{InskoYong}, and in some cases, the results of \cite{ADGH} give a set of generators for such ideals, which can be used to show reducedness. (Further details are in Section~\ref{sec:background}.)

Another motivation for the present paper is to introduce the theory of geometric vertex decomposability (GVD) to the study of the geometry of Hessenberg varieties. Geometric vertex decompositions were first defined and studied by Knutson, Miller, and Yong in their influential work \cite{KMY}, where they used their new theory to study Schubert determinantal ideals. More recently, the theory of geometric vertex decomposability, the definition of which is inherently inductive (recursive), has been linked to liaison theory and has been useful for understanding when a variety is in the Gorenstein liaison class of a complete intersections (i.e. ``glicci''). Briefly, geometric vertex decompositions can be a powerful tool for demonstrating the glicci property \cite{KleRaj}. 

This theory is relevant in our Hessenberg context because being GVD gives a convenient \emph{inductive} set-up for proving that a certain set of polynomials is a Gr\"obner basis. Indeed, in Section~\ref{sec: alternative}, we use these ideas to show that a certain set of generators form a Gr\"obner basis for local patch ideals for regular nilpotent Hessenberg varieties in the $w_0$-patch.

On the other hand, for the special case of the $w_0$-patch, which is our main focus in this manuscript, it turns out there is a way to build a Gr\"obner basis without using liaison theory which is both simpler and more general. By this we mean that our argument is valid whenever one has an ideal $I=\langle f_1, \ldots, f_n\rangle$ which is \textbf{a triangular complete intersection} in a sense we make precise in Definition~\ref{definition: lci triangular}.  The reader may then ask why we bother with the GVD theory;  the answer is that our simple argument for triangular complete intersections will not apply to the $w$-charts for $w \neq w_0$. Hence we have opted to present the argument using triangular complete intersections in Section~\ref{sec: background GVD} and Section~\ref{sec: fij} below, but we also present an alternative proof using geometric vertex decomposition in Section~\ref{sec: alternative}. At present, we expect that a full analysis of the general $w \neq w_0$ will need to rely on the GVD techniques.

We now turn to a discussion of our main results. Precise statements are in Corollary~\ref{cor: fkl are a GB} and Corollary~\ref{cor: Hess patch is GVD}. A \textbf{Hessenberg function} $h: [n] \to [n]$ is a function satisfying $h(i) \geq i$ for all $i$ and $h(1) \leq h(2) \leq \cdots \leq h(n)$. We refer to Definition~\ref{def: g v decomposable} for a precise definition of geometric vertex decomposability. The polynomials $f_{k,\ell}^{w_0}$ are defined in Definition~\ref{def: fkl}. 

\medskip

\noindent \textbf{Theorem.} Let $n$ be a positive integer with $n \geq 3$. Let $h: [n] \to [n]$ be a Hessenberg function satisfying $h(i) \geq i+1$ for all $1 \leq i\leq n-1$. Then the set of polynomials $\{f_{k,\ell}^{w_0}\}$ form a Gr\"obner basis for the Hessenberg patch ideal $I_{w_0,h}$ of $\hess(\mathsf{N},h)$ in the $w_0$-coordinate chart with respect to an appropriately chosen monomial order, and its initial ideal is an ideal of indeterminates. Furthermore, $I_{w_0,h}$ is geometrically vertex decomposable.

\medskip

We remark that, as mentioned above, we obtain the above theorem by first proving analogous results in the more general setting of triangular complete intersections. We expect these to also be of independent interest.

 In addition to our Gr\"obner basis results above, we initiate a study of Frobenius splittings in the context of Hessenberg varieties in Section~\ref{sec: frobenius}. It is known that there exists a Frobenius splitting of the flag variety $\flag(\mathbb{C}^n)$ which is compatible in a suitable sense with all Schubert and opposite Schubert varieties \cite{Brion-Kumar}; this has a geometric interpretation in terms of the anticanonical divisor class of $\flag(\mathbb{C}^n)$. Thus, it is natural to ask whether there is a Hessenberg analogue of this theory, namely, we may ask whether there exists a Frobenius splitting of $\flag(\mathbb{C}^n)$ which simultaneously compatibly splits all regular nilpotent Hessenberg varieties for (indecomposable) Hessenberg functions. It is known that a Frobenius splitting on an ambient variety restricts to a Frobenius splitting on an open dense affine coordinate chart, so if such a statement were true, then it would also hold true on a coordinate chart.  Our last main result in this manuscript is to show that for a specific and explicit choice of Frobenius splitting on the $w_0$-coordinate chart, this necessary condition holds. A precise version of what follows is contained in Corollary~\ref{cor: Hessenberg_split} and Corollary~\ref{cor: simultaneous split}. 

\medskip
\noindent \textbf{Theorem.} Let $p>0$ be a prime. There is an explicit Frobenius splitting $\varphi$ of the coordinate ring of the $w_0$-chart of $\flag(\mathbb{C}^n)$ with respect to which the local Hessenberg patch ideal $I_{w_0,h,p}$ is compatibly split. In particular, there is a partially ordered set (ordered by inclusion) of ideals $\{I_{w_0,h,p}\}$, indexed by the set of (indecomposable) Hessenberg functions $h$, which are simultaneously compatibly split with respect to $\varphi$.

\medskip
Again, we remark that our approach to the proof of the above theorem is to first prove an analogous general result for triangular complete intersections. 

Finally, in Section~\ref{sec: alternative}, we give the alternate proof of the Gr\"obner basis and GVD result for $I_{w_0,h}$ using liaison theory instead of relying on the specific setting of triangular complete intersections.

\medskip

Much of the discussion above has focused exclusively on the $w_0$-chart. It is natural to ask what happens to the other coordinate charts for $w \neq w_0$. We have some computational evidence that suggests that, for $w \neq w_0$, the restriction of the local Hessenberg patch ideals $I_{w,h}$ to the coordinates corresponding to the Schubert cell has computationally convenient properties. We also have preliminary evidence suggesting that there are conditions on $h$ and $w$ (and an appopriate choice of monomial order) such that the initial ideal of $I_{w,h}$ possesses a square-free monomial degeneration. (See also Remark~\ref{remark: other permutations}.) We expect to explore these questions further in future work.

\medskip

\noindent \textbf{Acknowledgements.} 
The first author was supported in part by a Natural Sciences and Engineering Research Council Postdoctoral Fellowship of Canada. 
The second author was supported in part by the Natural Sciences and Engineering Research Council Discovery Grant 2019-06567 and a Canada Research Chair Tier 2 Award. 
Both authors express gratitude to Patricia Klein and Jenna Rajchgot for many useful conversations and to Mike Cummings for his patient Macaulay 2 computations and for providing the examples which form the basis of Remark~\ref{remark: other permutations}. Finally, we thank the anonymous referee for many helpful substantive comments which significantly improved the paper.

\section{Background} 
\label{sec:background} 

In this section we briefly recall some background and notation necessary for the discussion that follows. 

\subsection{The flag variety $\flag(\mathbb{C}^n)$}\label{subsec: flag}
\label{subsec:flag background}

The full flag variety $\flag(\mathbb{C}^n)$ is the set of nested sequences of subspaces
$$
\flag(\mathbb{C}^n) := 
\{ V_{\bullet} = (0 \subsetneq V_1 \subseteq V_2 \subsetneq \cdots \subsetneq V_{n-1} \subsetneq V_n = \mathbb{C}^n) \, \mid \, \mathrm{dim}_{\mathbb{C}}(V_i) = i \}
$$ 
in $\mathbb{C}^n$. By representing $V_{\bullet}$ by an $n \times n$ matrix (whose leftmost $i$ many columns span $V_i$), we may identify $\flag(\mathbb{C}^n)$ as the homogeneous space $GL_n(\mathbb{C})/B$. Here, $B$ is the Borel subgroup of $GL_n(\mathbb{C})$ consisting of upper-triangular invertible matrices. Let $U^-$ denote the subgroup  in $GL_n(\mathbb{C})$ consisting of lower-triangular matrices with $1$'s along the diagonal.  Then 
$U^{-}B \subset GL_n(\mathbb{C})/B$ is the set of left cosets $uB$ with $u \in U^-$.  This is an open dense subset of $GL_n(\mathbb{C})/B \cong \flag(\mathbb{C}^n)$ and can be profitably viewed as a ``coordinate chart'' on $GL_n(\mathbb{C})/B$. 

Let $S_n$ denote the symmetric group on $n$ letters and $w \in S_n$ a permutation. We can identify $S_n$ with the set of permutation flags in $\flag(\mathbb{C}^n)$ and view it as a subgroup of $GL_n(\mathbb{C})$ by taking $w$ to the associated permutation matrix. By abuse of notation we will often denote by the same $w$ the element in $S_n$, its associated flag, and its associated permutation matrix.  Translating the coordinate chart $U^-B$ by multiplication by $w$ on the left, we can define 
\begin{equation}\labell{eq: def Nw}
\mathcal{N}_w := wU^{-}B \subseteq GL_n(\mathbb{C})/B,
\end{equation}
which is an open cell (i.e., a coordinate chart) in $GL_n(\mathbb{C})/B$ containing the permutation flag $w$.  (Note that our $\mathcal{N}_w$ are translates of the standard open dense Bruhat cell $Bw_0B/B$ in $GL_n(\CC)/B$.) 
It is well-known that $\flag(\mathbb{C}^n) \cong GL_n(\mathbb{C})/B$ can be covered by these $n!$ many coordinate charts, each centered around a permutation flag $w$. 

In fact, each $\mathcal{N}_w$ is isomorphic to a complex affine space of dimension $\frac{n(n-1)}{2}$. To see this, let 
\begin{equation}\label{eq: def M}
M := 
\begin{bmatrix} 
1 &  & & & \\
\star & 1 & & & \\
\vdots & \vdots & \ddots & & \\
\star & \star & \cdots & 1 & \\
\star & \star & \cdots & \star & 1 
\end{bmatrix} 
\end{equation}
denote an element in $U^-$ where the $\star$'s represent arbitrary complex numbers, and consider the map $M \mapsto wMB \in GL_n(\mathbb{C})/B$. Since $U^-\cap B =\{e\}$, it is not difficult to check that this defines an embedding $U^- \cong \mathbb{A}^{n(n-1)/2} \stackrel{\cong}{\rightarrow} \mathcal{N}_w \subset \flag(\mathbb{C}^n)$ parametrizing the coordinate chart $\mathcal{N}_w$. A point in $\mathcal{N}_w $ can be uniquely identified with the $w$-translate of an element $M$ in $U^-$, and thus a point in $\mathcal{N}_w$ is uniquely determined by a matrix $wM=(x_{i,j})$ satisfying
\[
x_{w(j),j}=1 \, \, \textup{ for } \, \, 
j \in [n] \, \, \textup{ and  } \, \,  x_{w(i),j}=0 \, \, \textup{ for } \, \, 
i, j \in [n], j>i.
\]
Thus the coordinate ring of $\mathcal{N}_w$, which we denote by $\mathbb{C}[\mathbf{x}_w]$, is isomorphic to the polynomial ring in the $n(n-1)/2$ variables not specified by the above relations. 

For instance, let $w_0$ be the Bruhat-longest element in $S_n$, 
so in one-line notation, 
 $$
 w_0 = [n \,\,  n-1 \,\,  n-2 \, \cdots \,  2\, \, 1].
 $$
Then for any positive integer $n$, the coordinate chart $\mathcal{N}_{w_0}$
can be parametrized by matrices of the form 
\begin{equation}\label{eq: generic w0 chart}
w_0M = 
\begin{bmatrix}
x_{1,1} & x_{1,2} & \cdots & x_{1,n-2} & x_{1,n-1} & 1 \\ 
x_{2,1} & x_{2,2} & \cdots & x_{2,n-2} & 1 & 0 \\
\vdots  &         &        &        &       \vdots & \vdots \\
x_{n-1,1}& 1 &    \cdots   & & 0 & 0 \\ 
1 & 0 &            \cdots  & &  0 & 0 &
\end{bmatrix}, 
\end{equation}
where we think of the variables $x_{i,j}$ in the matrix above as indeterminates (i.e., coordinates), taking values in $\mathbb{C}$. For a different choice of permutation $w$, these indeterminates will be located at different places within the matrix, but the idea is similar. 

\subsection{Regular nilpotent Hessenberg varieties}\label{subsec: reg nilp Hess}

Our notation and conventions largely follow the discussion in \cite{ADGH} so we will be brief. 
Let $n$ be a positive integer. We call a function $h: [n] := \{1,2,\cdots,n\} \to [n] := \{1,2,\cdots, n\}$ a \textbf{Hessenberg function} if it satisfies the conditions $h(i) \geq i$ for all $i$ and $h(i+1) \geq h(i)$ for $1 \leq i \leq n-1$. We say that a Hessenberg function is \textbf{indecomposable} if $h(i) \geq i+1$ for all $1 \leq i\leq n-1$. Let $A:\mathbb{C}^n\rightarrow \mathbb{C}^n$ be a linear operator and let $h: [n] \rightarrow [n]$ be an indecomposable Hessenberg function. Then we define the \textbf{Hessenberg variety associated to $A$ and $h$} to be the closed subvariety of $\flag(\mathbb{C}^n)$ given by 
\begin{equation}\label{eq: def Hess}
\hess(A,h):=\{ V_{\bullet} = (V_i) \in\flag(\mathbb{C}^n)\, \mid \, AV_i\subseteq V_{h(i)},\forall i\} \subset \flag(\mathbb{C}^n).   \end{equation} 
(Hessenberg varieties can be defined in more generality, in arbitrary Lie types, but for simplicity we restrict to Lie type A in this paper.) 

In this manuscript, we further focus on the special case when $A$ is a regular nilpotent operator. Specifically, define 
\begin{equation}\label{eq: def N} 
\mathsf{N} := 
\begin{bmatrix} 
0 & 1 & 0 & 0 & \cdots & 0 & 0  \\
0 & 0 & 1 & 0 & \cdots & 0 & 0 \\
0 & 0 & 0 & 1 & \cdots & 0 & 0 \\ 
\vdots &  & & &            &     & \vdots \\
0 & 0 & 0 & 0 &            & 0 & 1 \\ 
0 & 0 & 0 & 0 &            & 0 & 0  
\end{bmatrix} 
\end{equation} 
to be the matrix with $0$'s everywhere except the $1$'s immediately above the diagonal entries. (In other words, $N$ has a single Jordan block with eigenvalue $0$.) Hessenberg varieties $\hess(A,h)$ defined (as in~\eqref{eq: def Hess}) by $A$ which conjugate to $\mathsf{N}$ are called \textbf{regular nilpotent Hessenberg varieties}. In this case, we may restrict our attention to the case $A=\mathsf{N}$ since $\hess(A,h)\cong \hess(gAg^{-1},h)$ for $g\in GL_n(\mathbb{C})$. 

We now describe ``local defining equations'' for $\hess(\mathsf{N},h)$ following the method of \cite{ADGH}. By ``local'' we mean that for each choice of permutation $w \in S_n$ we focus on the local coordinate chart $\mathcal{N}_w \subseteq \flag(\mathbb{C}^n)$ centered at $w$ and ask for the defining equations for $\mathcal{N}_w \cap \hess(\mathsf{N},h)$ in the affine space $\mathcal{N}_w$. The method for deriving these equations is explained in detail in \cite[Section 3]{ADGH}, to which we refer the reader; here we will only briefly recall the results therein. Following \cite[Definition 3.3]{ADGH} we define certain polynomials $f_{k,\ell}^w$ in $\mathbb{C}[\mathbf{x}_w]$ as follows.

\begin{definition}\label{def: fkl} 
Let $w \in S_n$ and let $k,\ell\in [n]$ with $k>h(\ell)$. We define the polynomial $f_{k,\ell}^w \in \mathbb{C}[\mathbf{x}_w]$ by 
\[
f^w_{k,\ell} := ((wM)^{-1} \mathsf{N}(wM))_{k,\ell},
\]
where some matrix entries of the matrix $wM$ are viewed as variables, as described above. 
\end{definition}

We also define, using the polynomials $f_{k,\ell}^w$ defined above, the following ideals 
\begin{equation}\labell{eq: def Iw} 
I_{w,h} :=  \langle f_{k,\ell}^w \, \mid \, k > h(\ell) \rangle \subseteq \mathbb{C}[\mathbf{x}_w],
\end{equation} 
which we call \textbf{Hessenberg patch ideals}. In other words, $I_{w,h}$ is the ideal generated by the $(k,\ell)$-th matrix entries of $((wM)^{-1}\mathsf{N}(wM))$ where $k>h(\ell)$. 
Examples of $I_{w,h}$ are computed in \cite[Section 3]{ADGH}. We also have the following result from \cite{ADGH}. 

\begin{lemma} 
The ideal $I_{w,h}$ is the defining ideal of the affine variety $\hess(\mathsf{N},h) \cap \mathcal{N}_w$. In particular, it is radical. 
\end{lemma} 

\begin{remark}\label{lci}
It is shown in \cite[Lemma 3.1]{ADGH} that $\hess(\mathsf{N},h)$ is a local complete intersection.
Therefore, to show that $\hess(\mathsf{N},h)$ is reduced, it is enough to show that it is generically reduced, which can be checked in a single chart, e.g. the $w_0$-chart. (Each chart is open and dense in $\hess(\mathsf{N},h)$.) 

\end{remark} 

\begin{example}\label{n5_example}

Let $n=5$ and $w_0=[5\, 4\, 3\, 2\, 1]$. 
We can compute the matrix $(w_0M)^{-1}\mathsf{N}(w_0M)$ to obtain

\[
  \begin{bmatrix}
    0 & 0 & 0 & 0 & 0\\
    1 & 0 & 0 & 0 & 0\\
    f^{w_0}_{3,1} & 1 & 0 & 0 & 0\\
    f^{w_0}_{4,1} & f^{w_0}_{4,2} & 1  &  0 &  0 \\
    f^{w_0}_{5,1} & f^{w_0}_{5,2} & f^{w_0}_{5,3} & 1 & 0 \\
  \end{bmatrix}
\]
where the $f_{k,\ell}\in \mathbb{C}[\mathbf{x}_{w_0}]$ are defined by the following formulas: 
\begin{align*}
    & f^{w_0}_{5,1}= -x_{1,2}+x_{1,3}(x_{3,2}-x_{4,1})+x_{1,4}(x_{2,2}-x_{2,3}x_{3,2}+x_{2,3}x_{4,1}-x_{3,1})+x_{2,1}\\
    & f^{w_0}_{5,2}= -x_{1,3}+x_{1,4}(x_{2,3}-x_{3,2})+x_{2,2}\\
    & f^{w_0}_{5,3} = -x_{1,4}+x_{2,3}\\
    & f^{w_0}_{4,1} = -x_{2,2}+x_{2,3}(x_{3,2}-x_{4,1})+x_{3,1}\\
    & f^{w_0}_{4,2} =-x_{2,3}+x_{3,2}\\
    & f^{w_0}_{3,1} =-x_{3,2}+x_{4,1}. 
\end{align*}
Therefore, if $h_1 = (2,3,4,5,5)$ and $h_2 = (3,4,4,5,5)$, then we have 
\[I_{w_0,h_1} = \langle f^{w_0}_{3,1},f^{w_0}_{4,1},f^{w_0}_{4,2},f^{w_0}_{5,1},f^{w_0}_{5,2},f^{w_0}_{5,3}\rangle \]
and 
\[I_{w_0,h_2} = \langle f^{w_0}_{4,1},f^{w_0}_{5,1},f^{w_0}_{5,2},f^{w_0}_{5,3}\rangle. \]
\end{example}

In what follows, it will be useful to have inductive formulas for the polynomials $f_{k,\ell}^w$ which generate the ideals $I_{w,h}$. We go into more detail in Section~\ref{sec: fij} but it may be helpful to see an example here. In particular, there are some indexing conventions that need careful attention. Again we follow the exposition of \cite{ADGH}. Here, and for much of the remainder of the manuscript, we restrict to the simplest case, namely, the $w=w_0$ chart. 
We begin by recalling an example \cite[Example 3.13]{ADGH} which can serve to orient the reader. 

\begin{example} 
Let $n=4$ and $h=(3,3,4,4)$. The longest element of $S_4$ is the permutation $w_0=[4 \, 3 \, 2 \, 1]$. The coordinate ring of
  $\mathcal{N}_{w_0}$ is
  $$\CC [\mathbf{x}_{w_0}] \cong \CC
  [x_{1,1},x_{1,2},x_{1,3},x_{2,1},x_{2,2},x_{3,1}],$$ and a point in
  $\mathcal{N}_{w_0}$ is determined by a matrix
  \begin{equation*}
    w_0M=
    \begin{pmatrix}
      x_{1,1} & x_{1,2} & x_{1,3} & 1\\
      x_{2,1} & x_{2,2} & 1 & 0\\
      x_{3,1} & 1 & 0 & 0\\
      1 & 0 & 0 & 0
    \end{pmatrix}.
  \end{equation*}
  The inverse must then have
  the form
  \begin{align}\label{eq:yij}
    (w_0M)^{-1} =
    \begin{pmatrix}
      0 & 0 & 0 & 1\\
      0 & 0 & 1 & y_{3,1}\\
      0 & 1 & y_{2,2} & y_{2,1}\\
      1 & y_{1,3} & y_{1,2} & y_{1,1}
    \end{pmatrix}
  \end{align}
  for some $y_{i,j}$. Note that the indexing is such that $y_{i,j}$ is the $(n+1-i, n+1-j)$-th entry in the inverse matrix. It is possible to obtain expressions for the $y_{i,j}$ in terms of the $x_{i,j}$ by 
  starting from the matrix equality $(w_0M)^{-1}(w_0M)={\mathbf{1}}_{4 \times 4}$ (the $4 \times 4$ identity matrix) and
  comparing entries. For example,
  \begin{equation*}
    \begin{split}
      &y_{1,3} = -x_{1,3},\\
      &y_{1,2} = -x_{1,2} - y_{1,3} x_{2,2} = -x_{1,2} + x_{1,3}
      x_{2,2}.
    \end{split}
  \end{equation*}
Alternatively, the $y_{i,j}$ can also be expressed using the standard adjoint formula for inverses of matrices, and thus can be computed using certain minors of the original matrix $w_0M$. We will mainly stick to the latter point of view in the arguments that follow. 

\end{example}

In fact, the above discussion for the case $n=4$ readily generalizes to all $n$. Indeed, we have 
for general $n$ that
 \begin{align}\label{eq:general yij}
    (w_0M)^{-1} = 
    \begin{pmatrix}
      & & & & 1\\
      & & & 1 & y_{n-1,1}\\
      & &  & \vdots & \vdots\\
      & 1 & \ldots & y_{2,2} & y_{2,1}\\
      1 & y_{1,n-1} & \ldots & y_{1,2} & y_{1,1}
    \end{pmatrix}
  \end{align}
where again the $y_{i,j}$ are polynomials in the $\mathbf{x}_{w_0}$ variables. 
There is an inductive procedure to compute the $y_{i,j}$ which we won't recount in detail here, but some facts about the $y_{i,j}$ will be useful and are recorded below.

\begin{lemma}\label{lemma: y info ADGH}
Let $y_{i,j}$ denote the relevant entries of the inverse $(w_0M)^{-1}$ as above. 

\begin{enumerate} 
\item $y_{i,j}$ only depends on the variables $x_{i',j'}$ for $i' \geq i, j' \geq j$. 
\item $y_{i,j}=1$ if $i+j=n+1$. 
\item Suppose $i+j<n+1$. When expressed as a polynomial in the $\mathbf{x}_{w_0}$ variables, $y_{i,j}$ contains no constant terms, i.e., all monomials appearing in $y_{i,j}$ have degree $\geq 1$. 
\end{enumerate} 
\end{lemma} 

\begin{proof} 
The first claim is observed in \cite[cf. the discussion in proof of Lemma 3.12]{ADGH}. The second claim is also in \cite{ADGH}. The third claim follows from a straightforward induction argument which we now briefly sketch. Since $(w_0M)^{-1}(w_0M)={\mathbf{1}}_{n \times n}$ (the identity matrix), it is immediate that $y_{i,j} = -x_{i,j}$ if $i+j=n$. If $i+j<n$, it also follows from $(w_0M)^{-1}(w_0M)= {\mathbf{1}}_{n \times n}$ that $y_{i,j}$ is a polynomial in the variables $x_{i',j'}$ and $y_{i',j'}$ for $i'+j' \geq i+j$, which by a simple induction argument has no constant term. The base case $i+j=n$ yields the result. 
\end{proof}

The following formula for the $f_{k,\ell}^{w_0}$ from \cite{ADGH} will be useful in our later arguments. 

\begin{lemma}\label{lemma: fij formula} \cite[Equation (3.6), with different indexing conventions]{ADGH}
Let $k, \ell$ be such that $k > \ell$. Then 
\begin{equation}\label{eq: fkl in x and y}
f^{w_0}_{k,\ell} = x_{n+2-k, \ell} + 
\sum_{s=n+2-k}^{n-\ell} x_{s+1,\ell} \, y_{n+1-k, n+1-s}.
\end{equation}

\end{lemma}


\subsection{A torus action on $\hess(\mathsf{N},h)$ and $\CC[\mathbf{x}_{w_0}]$}\label{subsec: nonstandard} 


In order to apply some of the results from \cite{KleRaj} that relate liaison theory to geometric vertex decomposition in Section~\ref{sec: alternative}, we need a homogeneity condition. However, the $f_{k,\ell}^w$ are not in general homogeneous with respect to the standard grading on $\CC [\mathbf{x}_w]$. This turns out to not be a problem since there is a circle action on $\hess(\mathsf{N},h)$ which gives a non-standard grading of $\CC [\mathbf{x}_w]$ for which the $f_{k,\ell}^w$ are in fact homogeneous. 

We now describe this circle action on $\hess(\mathsf{N},h)$. Consider the circle subgroup $\mathsf{S} \cong \mathbb{C}^*$ of the maximal torus of $GL_n(\CC)$ 
\[
\mathsf{S} := \{ t:= \mathrm{diag}(g, g^2, \ldots, g^n) \, \mid \, g \in \CC^*\}. 
\]
It is straightforward to check that $\mathsf{S}$ preserves $\hess(\mathsf{N},h)$. In fact one can compute (for the diagonal matrix $t = \mathrm{diag}(g, \ldots, g^n)$)
that 
\[
t \mathsf{N} t^{-1} = g \mathsf{N},
\]
so the conjugation action becomes multiplication by the scalar $g$. We can also explicitly compute the action of $S$ on the local coordinate patch $\mathcal{N}_{w}$.  For concreteness, here we will focus on the $w_0$-chart, where the action is given as follows. The standard maximal torus action on $GL(n,\CC)/B$ is given by left multiplication on left cosets. More precisely, given a matrix $w_0M$ representing a flag $[w_0 M] \in GL(n,\CC)/B$, we have  
\[
t \cdot [w_0 M] = [t(w_0 M)].
\]
It is not difficult to compute $t(w_0M)$ directly. 
To read off the action in terms of the coordinate chart $\mathcal{N}_{w_0}$ we must now find a matrix $M'$ of the form~\eqref{eq: def M} such that $tw_0 M = w_0 M'$, and it is not hard to see from $[t(w_0 M)] = [t(w_0 M)\tilde{t}]$ for $\tilde{t} = \mathrm{diag}(g^{-n}, g^{-n+1},\cdots, g^{-1})$ that we obtain 
\[
w_0 M' = 
\begin{bmatrix}
g^{1-n}x_{1,1} & g^{2-n}x_{1,2} & \cdots & g^{-2}x_{1,n-2} & g^{-1}x_{1,n-1} & 1 \\ 
g^{2-n}x_{2,1} & g^{3-n}x_{2,2} & \cdots & g^{-1}x_{2,n-2} & 1 & 0 \\
\vdots  &         &        &        &       \vdots & \vdots \\
g^{-1}x_{n-1,1}& 1 &    \cdots   & & 0 & 0 \\ 
1 & 0 &            \cdots  & &  0 & 0 & \\
\end{bmatrix}. 
\]
This torus action induces an $\mathsf{S}$-action on $\CC [\mathbf{x}_w]$ given by $t\cdot x_{i,j} = g^{i+j-n-1}x_{i,j}$. We can use this action to define a (positive) $\mathbb{Z}$-grading on $\CC [\mathbf{x}_w]$ where a polynomial $f(\mathbf{x}_w)$ is homogeneous of degree $d\geq 0$ if 
\begin{equation}\label{eq: def nonstandard grading} 
g^d \, f(t\cdot \mathbf{x}_w)= f(\mathbf{x}_w).
\end{equation} 

Furthermore, since the $f_{k,\ell}^{w_0}$ are defined as entries of the matrix $(w_0M)^{-1}\mathsf{N}(w_0M)$, then $t\cdot f_{k,\ell}^{w_0}$ can be computed as the entries of the matrix \[(t(w_0M)\tilde{t})^{-1}\mathsf{N}(t(w_0M)\tilde{t}) = \tilde{t}^{-1}((w_0M)^{-1}g\mathsf{N}(w_0M))\tilde{t}\]
showing that $t\cdot f_{k,\ell}^{w_0} = g^{1+\ell-k}f_{k,\ell}^{w_0}$. The above discussion can be summarized as follows. 

\begin{lemma}\label{homogeneous}
The $f_{k,\ell}^{w_0}$ are homogeneous with respect to the non-standard positive $\mathbb{Z}$-grading of $\CC [\mathbf{x}_{w_0}]$ defined in~\eqref{eq: def nonstandard grading}. 
\end{lemma} 

\begin{remark}\label{remark: not positive}
A straightforward similar computation shows that in the other charts $\mathcal{N}_{w}$ for $w \neq w_0$, the $\mathsf{S}$-action also induces a nonstandard $\mathbb{Z}$-grading on $\CC[\mathbf{x}_w]$ with respect to which the ideal $I_{w,h}$ is homogeneous, and 
the generators $f_{k,\ell}^w$
are homogeneous.
However, in the general $w$ case with $w \neq w_0$, it will not necessarily be true
that this $\mathbb{Z}$-grading 
on $\CC[\mathbf{x}_w]$ 
is positive, i.e., it can happen that a non-constant element in $I_{w,h}$ may have degree zero. 
\end{remark}

\section{Geometric Vertex Decomposition and Gr\"obner bases for Triangular Complete Intersections}\label{sec: background GVD}

In this section, we prove some results concerning Gr\"obner bases and geometric vertex decomposition for certain complete intersection ideals which we call \textbf{triangular} (see Definition~\ref{definition: lci triangular}). Our main assertions are Theorem~\ref{theorem: lci triangular groebner} and Corollary~\ref{corollary: lci triangular gvd} below. We believe that these statements are well-known to experts, but we were unable to locate the statements in the literature, so we prove them here. Our motivation for these results stems from particular examples of Hessenberg varieties, but they are also of independent interest. Specifically, in Section~\ref{sec: fij} we use the general results below to construct Gr\"obner bases for the local Hessenberg patch ideals $I_{w_0,h}$ by showing that they are triangular complete intersections.

We begin by briefly recalling the notion of a geometric vertex decomposition. For this section, let $\mathbb{K}$ be an arbitrary field, and let $R$ denote a polynomial ring over $\mathbb{K}$ with a finite and fixed set of indeterminates which we denote by $\mathbf{x}$. (In our setting of the local defining ideals $I_{w_0,h}$ of regular nilpotent Hessenberg varieties, the set of indeterminates will be the set $\mathbf{x}_{w_0}$ as given in Section~\ref{subsec: reg nilp Hess}.) 
Now let $I$ be an ideal in $R$. Suppose $y \in \mathbf{x}$ is one of the indeterminates in $\mathbf{x}$. The \textbf{initial $y$-form} $\init_yf$ of $f \in R$ is the sum of all terms of $f$ having the highest power of $y$. In particular, if $y$ does not divide any term of $f$, then $\init_y(f)=f$. We say a monomial order $<$ on $R$ is \textbf{$y$-compatible} if it satisfies $\init_{<}(f) = \init_{<}(\init_y(f))$ for every $f \in R$. With respect to such a $y$-compatible monomial order, suppose $\mathcal{G} = \{y^{d_i} q_i + r_i \, \mid \, 1 \leq i \leq m\}$ is a Gr\"obner basis for $I$, where $y$ does not divide any $q_i$ and $\init_y(y^{d_i} q_i + r_i) = y^{d_i} q_i$. In this situation it is straightforward to see that $\init_y(I) = \langle y^{d_i} q_i \, \mid \, 1 \leq i \leq m \rangle$. We have the following. 

\begin{definition}\label{def: GVD} (\cite[Definition 2.3]{KleRaj})
In the setting above, define $C_{y,I} := \langle q_i \, \mid \, 1 \leq i \leq m \rangle$ and $N_{y,I} := \langle q_i \, \mid \, d_i = 0 \rangle$. If $\init_y(I) = C_{y,I} \cap (N_{y,I} + \langle y \rangle)$, then we call this decomposition a \textbf{geometric vertex decomposition of $I$ with respect to $y$}. A geometric vertex decomposition is \textbf{degenerate} if $C_{y,I} = N_{y,I}$ or if $C_{y,I} = \langle 1 \rangle$, and \textbf{non-degenerate} otherwise. 
\end{definition}

For further motivation and history surrounding these ideas see \cite{KleRaj}.  For the purposes of this paper it is important to have an inductive framework for geometric vertex decompositions, in the sense that the ideals $N_{y,I}$ and $C_{y,I}$ can also be equipped with such decompositions. This idea is made precise in Definition~\ref{def: g v decomposable} below. Recall that $I$ is said to be \textbf{unmixed} if $\dim(R/P) = \dim(R/I)$ for all $P \in \ass(I)$. In particular, the affine variety defined by an unmixed ideal $I$ will not contain any embedded components and will have equidimensional irreducible components.

\begin{definition}(\cite[Definition 2.6]{KleRaj})\label{def: g v decomposable}
Let $I$ be an ideal in $R$. We say $I$ is \textbf{geometrically vertex decomposable} (or GVD) if $I$ is unmixed and if 
\begin{enumerate} 
\item $I=\langle 1\rangle$, or, $I$ is generated by a (possibly empty) list of indeterminates, or,
\item for some fixed indeterminate $y$ of $R$, $\init_y(I) = C_{y,I} \cap (N_{y,I} + \langle y \rangle)$ is a geometric vertex decomposition and the contractions of $N_{y,I}$ and $C_{y,I}$ to $\mathbb{K}[\mathbf{x} \setminus y]$ are geometrically vertex decomposable. 
\end{enumerate} 
\end{definition}

From the point of view of Gr\"obner geometry, one motivation for asking for geometric vertex decomposability of a homogeneous ideal $I$ is that if such a decomposition exists with respect to a fixed monomial order $<$ of $R$, then there is an associated degeneration of $I$, from which it is possible to construct Gr\"obner bases of $I$. In a different direction, possessing a geometric vertex decomposition has consequences relating the ideals in the decomposition as in \cite[Theorem 2.1]{KMY}, such as providing a recursive formulation for the Hilbert series of $R/I$. In this manuscript, however, we take the point of view that if a set of generators of an ideal has a certain form with respect to a monomial order, then we can conclude both that the generators form a Gr\"obner basis, and that the ideal is GVD. 
We have the following.

\begin{definition}\label{definition: lci triangular} 
Let $\mathbb{K}$ be a field and $I \subseteq R=\mathbb{K}[x_1,\ldots, x_N]$ be an ideal and suppose that $\mathrm{ht}(I)=n$. Suppose there exists an ordered list of generators $\{f_1,\ldots, f_n\}$ of $I$ in $R$, a monomial order $<$ on $R$, and a list of indeterminates $\{x_{i_1}, x_{i_2}, \ldots, x_{i_n}\}$ such that 
\begin{enumerate} 
\item for each $j$ with $1 \leq j \leq n$, the initial term $\init_{<}(f_j)$ is a multiple $u_j x_{i_j}$ of $x_{i_j}$ for some $u_j \in \mathbb{K}^*$, and 
\item the indeterminate $x_{i_j}$ does not appear in any term of $f_m$ for $m>j$. 
\end{enumerate} 
Then we say that the ideal $I$ is a \textbf{triangular complete intersection of height $n$ with respect to $<$}. When the monomial order $<$ is understood from context, we say $I$ is a \textbf{triangular complete intersection of height $n$.} 
\end{definition}

\begin{example}
As a simple example, suppose $S=\mathbb{C}[x_1,x_2,x_3, x_4]$ and suppose we use the lexicographic order with $x_1 > x_2 > x_3 > x_4$. Suppose also that $f_1 = x_1 + x_2^2 + x_2 x_3$ and $f_2 = x_2 + x_3x_4$ and $f_3 = x_3 + x_4^2$. We claim that $I=\langle f_1, f_2, f_3 \rangle$ has height $3$ (this can be verified by Macaulay 2 or by noting that $\{f_1,f_2,f_3\}$ defines a regular sequence in $S$). It is straightforward to check that $\init_<(f_i)=x_i$ does not appear in any term of any $f_j$ with $j>i$, for $i=1,2,3$. Hence this is an example of a triangular complete intersection. 
\end{example}

\begin{theorem}\label{theorem: lci triangular groebner}
Suppose that $I\subset R=\mathbb{K}[x_1,\ldots,x_N]$ is a triangular complete intersection of height $n$ with respect to a monomial order $<$ on $R$. Let $\{f_1,\ldots,f_n\}$  be an ordered list of polynomials in $R$ satisfying the conditions of Definition \ref{definition: lci triangular}. Then $\init_<(I)$ is an ideal of indeterminates, and $\{f_1,\ldots,f_n\}$ defines a Gr\"obner basis of $I$ with respect to $<$. 
\end{theorem}

\begin{proof} 
By assumption, since they are (non-zero constant multiples of) distinct indeterminates, $\init_<(f_i)$ and $\init_<(f_j)$ are relatively prime for all $i\neq j$.  By \cite[Chapter 2.9, Proposition 4]{CLO}, the $S$-polynomials $S(f_i,f_j)$ reduce to zero for all $i \neq j$, so by \cite[Chapter 2.9, Theorem 3]{CLO}, $\{f_1,\ldots,f_n\}$ is a Gr\"obner basis with respect to $<$ for the ideal it generates. In particular, it is immediate that $\init_<(I)$ is an ideal of indeterminates. 
\end{proof} 

In addition to obtaining Gr\"obner bases, we can also conclude that a triangular complete intersection ideal is GVD, because -- as we have just seen -- its initial ideal is generated by indeterminates. To justify this, we need the notion of an ideal being $<$-compatibly geometric vertex decomposable \cite[Definition 2.11]{KleRaj}. This concept is similar to Definition \ref{def: g v decomposable} but different in the following way. Note that in the definition of geometric vertex decomposability, the monomial order $<$ with respect to which $N_{y,I}$ and $C_{y,I}$ are also GVD (in part (b) of Definition~\ref{def: g v decomposable}) is not specified, so in particular, they may vary. However, for an ideal which is $<$-compatibly GVD in the sense of \cite[Definition 2.11]{KleRaj}, the monomial order $<$ is fixed once and for all, and the contractions of $N_{y,I}$ and $C_{y,I}$ are required to also be $<$-compatibly GVD with respect to the induced monomial order of the fixed $<$ on the smaller ring $\mathbb{K}[\mathbf{x}\setminus y]$. We have the following. 

\begin{lemma}\label{lem:initGVD}
Let $I \subseteq R = \mathbb{K}[x_1,\ldots, x_N]$ be an ideal and $<$ a lexicographic monomial order on $R$.
If $\init_<(I)$ is an ideal of indeterminates, then $I$ is geometrically vertex decomposable.
\end{lemma}

\begin{proof}
The Stanley-Reisner complex associated to an ideal of indeterminates is a single (possibly empty) simplex. In particular, it is vertex decomposable in the classical sense of simplicial complexes. In fact, it is not hard to see that it is $<$-compatibly geometrically vertex decomposable. Now the result immediately follows from \cite[Proposition 2.14]{KleRaj} (which is valid over an arbitrary field $\mathbb{K}$).
\end{proof}

The following is immediate. 

\begin{corollary}\label{corollary: lci triangular gvd} 
If $I \subset R = \mathbb{K}[x_1,\ldots, x_N]$ is a triangular complete intersection of height $n$ with respect to a lexicographic order $<$, then it is GVD.
\end{corollary}

\begin{proof} 
We saw in Theorem~\ref{theorem: lci triangular groebner} that if $I$ is a triangular complete intersection of height $n$ with respect to $<$, then $\init_<(I)$ is an ideal of indeterminates. Now the result follows from Lemma~\ref{lem:initGVD}. 
\end{proof} 

\begin{remark}\label{remark: complete intersection}
It turns out that an ideal $I$ which is generated by $\{f_1,\ldots,f_n\}$ satisfying conditions $(1)$ and $(2)$ of Definition \ref{definition: lci triangular} is automatically a complete intersection. Indeed, by Lemma~\ref{lem:initGVD}, $\init_<(I)$ is an ideal of indeterminates and hence a complete intersection, so it follows that $I$ is also a complete intersection \cite[Corollary 19.3.8]{EGA}. Additionally, if $I$ is a homogeneous ideal (with respect to any grading of $R$), then we could also conclude that the height of $I$ and $\init_<(I)$ are equal by \cite[Section 8.2.3 and Theorem 15.26]{Eisenbud}. Homogenization of $I$ (see \cite[Section 1.8]{Eisenbud} for details) would yield a similar result for non-homogeneous ideals. 
\end{remark}

\section{The case of the regular nilpotent Hessenberg variety}\label{sec: fij}

The goal of this section is to show that the polynomials $f_{k,\ell}^{w_0}$ defining the Hessenberg patch ideal $I_{w_0,h}$ (i.e., the ideal of local defining equations for the Hessenberg variety in the $w_0$-chart) obey certain recursive relationships. These observations will allow us to show that the patch ideal $I_{w_0,h}$ is a triangular complete intersection in the sense of Definition~\ref{definition: lci triangular}. From this, we obtain our main results, namely, that 
the $\{f^{w_0}_{k,\ell}\}$ form a Gr\"obner basis for $I_{w_0,h}$ (Corollary~\ref{cor: fkl are a GB}) and that $I_{w_0,h}$ is GVD (Corollary~\ref{cor: Hess patch is GVD}).

\subsection{Recursively defining $f^{w_0}_{k,\ell}$}\label{subsec: recursive def of f}

In this section, we work solely in the $w_0$-chart. For this reason, and for notational simplicity, in this section we denote $f_{k,\ell}^{w_0}$ by $f_{k,\ell}$ and $w_0M$ by $M'$.

We first note that the matrix entries $y_{i,j}$ of $(M')^{-1}$ in~\eqref{eq:general yij} can be expressed in terms of cofactors, via the well-known formula for matrix inverses. 
Indeed, let $M'_{n+1-j, n+1-i}$ denote the $(n-1)\times (n-1)$ matrix obtained 
from the $n \times n$ matrix $M'$ in~\eqref{eq: generic w0 chart} by deleting the $(n+1-j)$-th row and $(n+1-i)$-th column. 
The following lemma is immediate. 

\begin{lemma}\label{lemma: y minor} 
\[
y_{i,j} = (-1)^{n(n-1)/2} (-1)^{i+j} \det M'_{n+1-j, n+1-i}.
\]
\end{lemma} 

Here the $(-1)^{n(n-1)/2}$ represents the determinant of $M'$, and the factor $(-1)^{i+j}$ is the sign that comes from the cofactors in the standard matrix inverse formula.

Now recall that we visualize the polynomial $f_{k,\ell}$ as being associated to the $(k, \ell)$-th entry of the matrix $(M')^{-1} \mathsf{N}M'$. Our next observation is that the $f_{k,\ell}$ which lie immediately below the main diagonal are particularly simple. Note that these particular $f_{k,\ell}$ are not generators for $I_{w_0,h}$ when $h$ is indecomposable, but they do appear in the recursive expression of Proposition \ref{prop: fij inductive}, which is why this lemma is useful as a base case. 

\begin{lemma}\label{lemma: f j+1 j}
Let $\ell \in \mathbb{Z}$ with $1 \leq \ell \leq n-1$. Then $f_{\ell+1,\ell}=1$. 
\end{lemma}

\begin{proof} 
From Lemma~\ref{lemma: fij formula} and by appropriate substitutions of variables we know that 
\begin{equation}\label{eq: off main diag}
    f_{\ell+1,\ell} = x_{n-\ell+1,\ell} + \sum_{s=n-\ell+1}^{n-\ell} x_{s+1,\ell} \, \, y_{n-\ell, n+1-s}. 
\end{equation}
We can see that the summation in the second term of the RHS of~\eqref{eq: off main diag} is actually an empty sum, so $f_{\ell+1,\ell} = x_{n-\ell+1,\ell}$. We also know from the form of the matrix $M'$ that $x_{n-\ell+1,\ell}=1$, completing the proof. 
\end{proof} 

The next result gives an important inductive relationship between different $f_{k,\ell}$. This is our main technical tool for the remainder of the arguments concerning $I_{w_0,h}$. The point of the formula is that a given generator $f_{k,\ell}$ with $k > \ell+1$ can be expressed in terms of certain $f_{s,\ell}$'s with $\ell+1 \leq s \leq k-1$.

\begin{remark} 
In fact, we do not actually use the full strength of this recursion result in the current paper. It would be interesting to explore the recursion further, especially for other $w$-charts with $w \neq w_0$. We leave this open for future work. 
\end{remark}

\begin{proposition}\label{prop: fij inductive} 
Let $n \in \mathbb{Z}, n \geq 3$. For $1 \leq k,\ell \leq n$ and $k >\ell+1$, let $f_{k,\ell}$ be the $(k,\ell)$-th entry of the matrix $(M')^{-1}\mathsf{N}M'$ as above, considered as an element of $\CC[\mathbf{x}_{w_0}]$.  Then 
\begin{equation}\label{eq: inductive fij}
f_{k,\ell} = x_{n+2-k,\ell} - \left( \sum_{p=\ell+1}^{k-1} x_{n+1-k,p}f_{p,\ell} \right).
\end{equation}
\end{proposition} 

Before proving Proposition~\ref{prop: fij inductive} we introduce some notation. In the exposition above, we denoted the $(n-1) \times (n-1)$ minor of $M'$ obtained by deleting the $s$-th row and $p$-th column by the symbol $M'_{s,p}$. Here and below we denote the $(n-2) \times (n-2)$ minor of $M'$, obtained by deleting the $s$-th and $s'$-th rows and the $p$-th and $p'$-th columns, by the symbol $M'_{\{s,s'\},\{p,p'\}}$, where we assume that $s \neq s', p \neq p'$. Using this notation, we can further expand the determinant $\det M'_{s,n}$ as follows.  Suppose for instance that $s \geq 2$. In this case, we know that the matrix $M'_{s,n}$ has a top row of the form $(x_{1,1}, x_{1,2}, \cdots, x_{1,n-1})$. We may compute $\det M'_{s,n}$ by expanding along this top row, obtaining 
\begin{equation}\label{eq: det Mks}
\det M'_{s,n} = \sum_{p=1}^{n-1} (-1)^{p+1} x_{1,p} \det M'_{\{1,s\},\{p,n\}}.
\end{equation}
We have the following. 

\begin{lemma}\label{lemma: zero minors}
Let $n \in \mathbb{Z}, n \geq 3$, and $s,p \in \mathbb{Z}$ with $2 \leq s \leq n$ and $1 \leq p < n$. If $s+p \leq n$ then $\det M'_{\{1,s\},\{p,n\}} = 0$. 
\end{lemma} 

\begin{proof} 
Consider the $p\times (n-2)$ submatrix $A$ of $M'_{\{1,s\},\{p,n\}}$ consisting of the bottom $p$ many rows. Since $s\leq n-p$ by assumption, the bottom $p$ rows of $M'_{\{1,s\},\{p,n\}}$ are the same as the bottom $p$ rows of $M'$ with the $p$-th and $n$-th columns removed. From this point of view, it is easier to see that the top row of $A$ does not contain an entry of $1$, and only contains indeterminates and $0$'s. Moreover, this top row is clearly a linear combination of the $p-1$ many rows below it (when $p=1$, $A$ is just the zero row matrix). Since $M'_{\{1,s\},\{p,n\}}$ contains $p$ rows which are linearly dependent, $\det M'_{\{1,s\},\{p,n\}}=0$ as claimed. 
\end{proof}

With these preliminaries, we can prove the Proposition. 

\begin{proof}[Proof of Proposition~\ref{prop: fij inductive}]
We first prove the case $k=n$. 
From Lemmas~\ref{lemma: fij formula} and~\ref{lemma: y minor} and~\ref{lemma: zero minors}, it follows that  
\[
f_{n,\ell} = x_{2,\ell} + (-1)^{n(n-1)/2} \sum_{s=2}^{n-\ell} (-1)^{n-s} x_{s+1,\ell} \left( \sum_{p=1}^{n-1} (-1)^{p+1} x_{1,p} \det M'_{\{1,s\},\{p,n\}} \right). 
\]
On the other hand, from Lemma~\ref{lemma: zero minors} we know that some of the $\det M'_{\{1,s\},\{p,n\}}$ are equal to $0$. Hence we obtain 
\begin{equation}\label{eq: fnj second formula}
    \begin{split}
    f_{n,\ell} &= x_{2,\ell} + (-1)^{n(n-1)/2} \sum_{s=2}^{n-\ell} \sum_{p=n+1-s}^{n-1} (-1)^{n-s}(-1)^{p+1} x_{1,p} x_{s+1,\ell} \det M'_{\{1,s\},\{p,n\}} \\
    & = x_{2,\ell} + (-1)^{n(n-1)/2} \sum_{p=\ell+1}^{n-1} \sum_{s=n+1-p}^{n-\ell} (-1)^{n-s} (-1)^{p+1} x_{1,p} x_{s+1,\ell} \det M'_{\{1,s\},\{p,n\}} \\
    & = x_{2,\ell} + \sum_{p=\ell+1}^{n-1} x_{1,p} \left( (-1)^{n(n-1)/2}\sum_{s=n+1-p}^{n-\ell} (-1)^{p+1} (-1)^{n-s} x_{s+1,\ell} \det M'_{\{1,s\},\{p,n\}} \right)
    \end{split}
\end{equation}
where the second equality follows from reorganizing the summation of the indices $s$ and $p$. 
Our next step is to analyze the expressions $\det M'_{\{1,s\},\{p,n\}}$ appearing in~\eqref{eq: fnj second formula}. Since $p \leq n-1$ and $s \geq 2$, it follows that the $(n-1) \times (n-1)$ minor $M'_{s,p}$ of $M'$ has as its rightmost column the standard basis vector $(1,0,0,\cdots,0)^T$. Expanding $\det M_{k,s}$ along this column we see that 
\[
\det M'_{s,p} = (-1)^n \det M'_{\{1,s\},\{p,n\}}.
\]
Thus we can rewrite~\eqref{eq: fnj second formula} as 
\[
f_{n,\ell} = x_{2,\ell} + \sum_{p=\ell+1}^{n-1} x_{1,p} \left( (-1)^{n(n-1)/2} \sum_{s=n+1-p}^{n-\ell} (-1)^{s+p+1} x_{s+1,\ell} \det M'_{s,p} \right).
\]
In the above expression, we now analyze the term corresponding to $s=n+1-p$. Note that the $(n+1-p,p)$-th entry in the matrix $M'$ lies on the main antidiagonal and is equal to $1$. Deleting the $(n+1-p)$-th row and $p$-th column from $M'$ we see that the minor $M'_{n+1-p,p}$ is again a matrix with $1$'s along the main antidiagonal and $0$'s below it. This means 
\[
\det M'_{n+1-p,p} = (-1)^{(n-1)(n-2)/2}
\]
so we have $(-1)^{n(n-1)/2}(-1)^{n+2}x_{n+2-p,\ell}(-1)^{(n-1)(n-2)/2} = -x_{n+2-p,\ell}(-1)^{(n-2)(n+1)} = -x_{n+2-p,\ell} $. Therefore
\[
f_{n,\ell} = x_{2,\ell} + \sum_{p=\ell+1}^{n-1} x_{1,p} \left( -x_{n+2-p,\ell} + (-1)^{n(n-1)/2} \sum_{s=n+2-p}^{n-\ell} (-1)^{s+p+1} x_{s+1,\ell} \det M'_{s,p} \right),
\]
which is equal to  
\[
x_{2,\ell} -\sum_{p=\ell+1}^{n-1} x_{1,p} \left(x_{n+2-p,\ell} + (-1)^{n(n-1)/2} \sum_{s=n+2-p}^{n-\ell} (-1)^{s+p} x_{s+1,\ell} \det M'_{s,p} \right)
\]
after factoring out $-1$. The portion in brackets is exactly $f_{p,\ell}$ by Lemmas~\ref{lemma: fij formula} and~\ref{lemma: y minor}, so we have proven the desired result for the case $k=n$. 

Now we wish to prove the case for general $k$. From the form of the matrix $M'$ and its inverse $(M')^{-1}$, it is clear that the upper-right square submatrix of $(M')^{-1}$ of size $k \times k$ for $k<n$ is an inverse to the lower-left $k \times k$ submatrix of $M'$. It follows that the upper-left $k \times k$ submatrix of $(M')^{-1}\mathsf{N} M'$ can be identified, upon suitable re-naming of coordinates, with the polynomials $f_{a,b}$ which would appear in the construction for Hessenbergs in $\flag(\mathbb{C}^k)$ where $k<n$. Thus, applying our arguments above for $f_{n,\ell}$ except with $k$ replacing the value $n$, we obtain the desired result by induction. 
\end{proof} 

\begin{example}\label{remark: fnj induction}
Explicitly, for the case $k=n$, 
Proposition~\ref{prop: fij inductive} says that 
\[
f_{n,\ell} = - x_{1,n-1} f_{n-1,\ell} -  x_{1,n-2} f_{n-2,\ell} - x_{1,n-3} f_{n-3,\ell} - \cdots - x_{1,\ell+1} f_{\ell+1,\ell} + x_{2,\ell}. 
\]
\end{example}


\subsection{$I_{w_0,h}$ is a triangular complete intersection}

We now use the inductive description of the $f_{k, \ell}$'s obtained in Section~\ref{subsec: recursive def of f} to prove our main results. In fact, by the results in Section~\ref{sec: background GVD}, all we need to show is that a certain choice of ordering on $f^{w_0}_{k,\ell}$, together with a certain choice of monomial order on $\CC[\mathbf{x}_{w_0}]$, satisfy the conditions to make $I_{w_0,h}$ a triangular complete intersection in the sense of Definition~\ref{definition: lci triangular}. This is what we do below.

Let $n$ be fixed. We begin with a definition of a specific monomial order on $\CC[\mathbf{x}_{w_0}]$.  We have the following.

\begin{definition}\label{def: order}
We define a monomial order $<_n$ on $\CC [\mathbf{x}_{w_0}]$ as follows: first, we order the variables in $\mathbf{x}_{w_0}$ by 
$$x_{i,j} >_n x_{i',j'} \hspace{2mm} \textup{ if } \hspace{2mm} i<i', \hspace{2mm} \textup{ or, if} \hspace{2mm} i=i' \hspace{2mm} \textup{ and } \hspace{2mm} j<j'.
$$
Then, we define $<_n$ to be the lexicographic ordering on monomials induced by the above ordering on the variables.
\end{definition}

In other words, the leading term is the one with the highest power of $x_{1,1}$, followed by the highest power of $x_{1,2}$ to break ties, followed by (in order, reading from left to right) the other variables in the top row of $w_0M$, so $x_{1,1} >_n x_{1,2} >_n x_{1,3} >_n \cdots$, followed by the variables (also reading from left to right) in the second row, and so on. For example, when $n=4$, then the ordering of the variables is $x_{1,1} >_4  x_{1,2} >_4 x_{1,3} >_4 x_{2,1} >_4 x_{2,2} >_4 x_{3,1}.$ 

To determine whether $I_{w_0,h}$ is a triangular complete intersection, we also need an ordering on a chosen set of generators.  For this purpose, we will sequentially order the polynomials $f^{w_0}_{k,\ell}$ in the following way:  
\begin{equation}\label{eq: fij sequence} f^{w_0}_{n,1}, f^{w_0}_{n,2},\cdots, f^{w_0}_{n,n-1}, f^{w_0}_{n-1,1}, f^{w_0}_{n-1,2}, \cdots 
\end{equation}
i.e. when we visualize the $f^{w_0}_{k,\ell}$ as entries in a matrix,  we ``read from the bottom row to the top row, and left to right along rows''. 
This ordering of the polynomials, together with the choice of monomial order $<_n$, makes $I_{w_0,h}$ a triangular complete intersection. We have the following.

\begin{lemma}\label{lemma: initial term and one variable}
Let $k, \ell \in \mathbb{Z}, 1 \leq k, \ell \leq n$ and $k > \ell+1$. Then: 
\begin{enumerate} 
\item With respect to the monomial order $>_n$ as above, we have 
$$
\init_{<_n}(f^{w_0}_{k,\ell}) = - x_{n+1-k, \ell+1}.
$$
\item The variable $x_{n+1-k,\ell+1}$  does \emph{not} appear in any polynomial $f^{w_0}_{k',\ell'}$ appearing after $f^{w_0}_{k,\ell}$ in the sequence~\eqref{eq: fij sequence}. 

\item The variable $x_{n+1-k, \ell+1}$ appears precisely once in the polynomial $f^{w_0}_{k,\ell}$, and all other variables $x_{i,j}$ which appear in $f^{w_0}_{k,\ell}$ have the property that either $i=n+1-k$ and $j > \ell+1$, or, $i>n+1-k$.
\end{enumerate} 
\end{lemma} 

\begin{remark}\label{remark: is lci} 
It follows from Remark \ref{remark: complete intersection} that the height of $I_{w_0,h}$ is equal to the number of $f_{k,\ell}^{w_0}$ generating it.  The claims (1) and (2) of Lemma~\ref{lemma: initial term and one variable} therefore imply, according to Definition~\ref{definition: lci triangular}, that $I_{w_0,h}$ is a triangular complete intersection.  
\end{remark} 

\begin{proof}[Proof of Lemma~\ref{lemma: initial term and one variable}]
From Proposition~\ref{prop: fij inductive} we see that the summand in the expression on the RHS of~\eqref{eq: inductive fij} corresponding to the index $p=\ell+1$ is of the form $-x_{n+1-k, \ell+1} \cdot f^{w_0}_{\ell+1,\ell}$.  From Lemma~\ref{lemma: f j+1 j} we know that $f^{w_0}_{\ell+1,\ell}=1$, so the $p=\ell+1$ summand is in fact exactly 
$- x_{n+1-k, \ell+1}$, the claimed initial term. To show that this is indeed the initial term, it would suffice to see that both $x_{n+2-k,\ell}$ - the first expression in the RHS of~\eqref{eq: inductive fij} -- and all summands corresponding to $\ell+2 \leq p \leq k-1$ contain only variables $x_{i,j}$ satisfying either $i > n+1-k$, or, $i=n+1-k$ and $j>\ell+1$. By definition of the monomial order $<_n$, this would mean that all other variables appearing are strictly less than $x_{n+1-k, \ell+1}$ and hence the initial term as claimed. First, the variable $x_{n+2-k,\ell}$ has first index $n+2-k>n+1-k$, so $x_{n+2-k,\ell}<x_{n+1-k, \ell+1}$ as desired. Next we analyze the expression $x_{n+1-k,p} \cdot f_{p,\ell}$ for $\ell+2 \leq p \leq k-1$. Since $p>\ell+1$ by assumption, $x_{n+1-k,p} < x_{n+1-k,\ell+1}$ as desired. Finally, we analyze the variables appearing in $f_{p,\ell}$ for $\ell+2\leq p \leq k-1$. From Lemma~\ref{lemma: fij formula}, equation~\eqref{eq: fkl in x and y}  (replacing $k$ with $p$), and Lemma~\ref{lemma: y info ADGH}(1), it is then straightforward to see that the variables $x_{i,j}$ appearing in $f_{p,\ell}$ all satisfy $i > n+1-k$, so are less than $x_{n+1-k,\ell+1}$ as desired. This proves (1). 

The statement (3) follows immediately from the proof of (1). So it remains to show the statement (3). We have just seen that for a given $f^{w_0}_{k,\ell}$, the only variables $x_{i,j}$ which appear are $x_{n+1-k,\ell+1}$ and those which are smaller than $x_{n+1-k,\ell+1}$ with respect to $<_n$. For any $f^{w_0}_{k',\ell'}$ with either $k'=k$ and $\ell'>\ell$, or, $k'<k$, it is straightforward to see that $x_{n+1-k,\ell+1}$ cannot appear. This proves (3). 
\end{proof}

It now follows easily that we have a Gr\"obner basis. The precise statement is below.

\begin{corollary}\label{cor: fkl are a GB}
Let $n \in \mathbb{Z}, n \geq 3$.  Let $h: [n] \to [n]$ be an indecomposable Hessenberg function. Let $\{f^{w_0}_{k, \ell} \, \mid \, k > h(\ell)\}$ be the generators of $I_{w_0,h}$ in $\mathbb{C}[\mathbf{x}_{w_0}]$. 
Then the $\{f^{w_0}_{k,\ell} \, \mid \, k > h(\ell)\}$ form a Gr\"obner basis for $I_{w_0,h}$ with respect to the monomial order $<_n$ of Definition~\ref{def: order}. 
\end{corollary} 

\begin{proof} 
Lemma~\ref{lemma: initial term and one variable} and Remark~\ref{remark: is lci} shows that $I_{w_0,h}$ is a triangular complete intersection with respect to the subsequence of generators $f_{k,\ell}^{w_0}$ of
\[f^{w_0}_{n,1}, f^{w_0}_{n,2},\cdots, f^{w_0}_{n,n-1}, f^{w_0}_{n-1,1}, f^{w_0}_{n-1,2}, \cdots \]
which satisfy $k>h(\ell)$, and our choice of monomial order $<_n$. The claim follows immediately from Theorem~\ref{theorem: lci triangular groebner}.

\end{proof}

We can also conclude that $I_{w_0,h}$ is GVD.

\begin{corollary}\label{cor: Hess patch is GVD}
In the same setting as above, $I_{w_0,h}$ is geometrically vertex decomposable. 
\end{corollary}

\begin{proof} 
Since $I_{w_0,h}$ is a triangular complete intersection with respect to $<_n$, the claim follows immediately from Corollary~\ref{corollary: lci triangular gvd}. 
\end{proof}

\begin{remark}\label{remark: other permutations}
The results above show that the initial ideal of $I_{w_0,h}$ with respect to the chosen monomial order is an ideal of indeterminates. In particular, it is a square-free monomial ideal.  While we do not expect this to hold true in all $w$-charts for $w \neq w_0$, we find it natural to ask the following: for which pairs $(h,w)$ of Hessenberg functions and permutations $w$ is it true that there exists a choice of monomial order $<$ on $\mathbb{C}[\mathbf{x}_w]$ such that $\init_{<}(I_{w,h})$ is a square-free monomial ideal (i.e. a Stanley-Reisner ideal)? 
Preliminary results from some Macaulay2 computations (for cases with small values of $n$) tentatively suggest that for any Hessenberg function $h$, there do exist some $w\neq w_0$ and some choices of monomial order such that the initial ideal of $I_{w,h}$ is square-free. For instance, when $n=4$, we have found the following. We say that a permutation is $[3,2,1]$-embedding if $w = [a_1 \, \, a_2 \, \, a_3 \, \, a_4]$ in one-line notation where there exists some choice of $1 \leq i < j<k \leq 4$ such that $a_i > a_j > a_k$. Our numerical explorations for this case have shown that if $w$ is $[3,2,1]$-embedding, then there exists a lexicographic monomial order $<$ on $\mathbb{C}[\mathbf{x}_w]$ such that $\init_{<}(I_{w,h})$ is a square-free monomial ideal. 
We leave further exploration of the general case to future work.  
\end{remark}

\section{Frobenius splittings and triangular complete intersections}\label{sec: frobenius}

In this section, we provide an application of our results on triangular complete intersections to the study of Frobenius splittings. Since we saw in the previous section that Hessenberg patch ideals for the $w_0$-patch are triangular complete intersections, our observations immediately apply also to this case. 

 Before stating our results, we take a moment to recall some of the context and motivation. 
Recall that the notion of a Frobenius splitting is defined in the setting of schemes defined over characteristic $p>0$; however, they are also useful 
in the study of schemes in characteristic $0$. 
For instance, Frobenius splittings were used by Brion and Kumar in \cite{Brion-Kumar} to prove that Schubert varieties in $G/B$ are reduced, normal, and Cohen-Macaulay. This is also why we are interested in Frobenius splittings in this manuscript.

One motivation for the considerations in this section is the following. It is known that there exists a Frobenius splitting of $G/B$ which compatibly splits all Schubert and opposite Schubert varieties \cite[Theorem 2.2.5]{Brion-Kumar}. In fact, the union of the Schubert and opposite Schubert divisors defines an anticanonical divisor of $G/B$, and this is related to the fact that the above Frobenius splitting is defined by sections of the $(p-1)$-st power of an anticanonical bundle \cite[Section 1.3]{Brion-Kumar}. It is natural to ask whether there is an analogous statement that remains true for Hessenberg varieties, leading us to the following. 

\begin{question}\label{Frobenius_question}
Does there exist a Frobenius splitting of $G/B$ which simultaneously compatibly splits all of the Hessenberg varieties $\hess(\mathsf{N},h)$ for indecomposable $h$? If so, can we construct an explicit such Frobenius splitting? 
\end{question}

Unfortunately, the approach to the above question regarding $G/B$ and its Schubert and opposite Schubert varieties as explained in \cite{Brion-Kumar} (which synthesizes results from various authors), constructs Frobenius splittings using Bott-Samelson varieties -- a technique which does not easily generalize to our setting of Hessenberg varieties. On the other hand, there is also a \emph{local} theory of Frobenius splittings, as discussed by Knutson in \cite{Knut, Knut2}; this local theory is better suited for studying our local patches $\hess(\mathsf{N},h)\cap \mathcal{N}_{w_0}$. This is the route we take below. In particular, a Frobenius splitting will restrict to open sets, in the sense that if there is a Frobenius splitting on an ambient space, then it must restrict to a Frobenius splitting on an open dense affine coordinate chart. Hence, a \emph{necessary} condition for Question \ref{Frobenius_question} to have a positive answer is that there exists a Frobenius splitting of $\mathcal{N}_{w_0}$ which simultaneously compatibly splits $\hess(\mathsf{N},h) \cap \mathcal{N}_{w_0}$ for all indecomposable $h$. We answer this positively in Corollaries~\ref{cor: Hessenberg_split} and~\ref{cor: simultaneous split} below.

We now shift our attention to the results we prove in this section. Our approach is similar to that of Sections~\ref{sec: background GVD} and~\ref{sec: fij}; namely, we first prove general statements about triangular complete intersections in Section~\ref{subsec: Frob of triangular lci}, and then apply these general results to the Hessenberg setting in Section~\ref{subsec: Hess Frob}. More precisely, let $R = \mathbb{K}[x_1,\ldots,x_N]$ be a polynomial ring over $\mathbb{K}$ a field of positive characteristic $p$ and suppose $R$ is equipped with a monomial order $<$. Suppose that an ideal $I \subseteq R$ is a triangular complete intersection with respect to the given monomial order. 

In this setting, in Section~\ref{subsec: Frob of triangular lci} we first construct an \emph{explicit} Frobenius splitting $\varphi_f$ of $I$. Secondly, we show that, with respect to this explicit Frobenius splitting $\varphi_f$, there is a natural family of compatibly split varieties, with respect to our Frobenius splitting $\varphi_f$. These compatibly split ideals are naturally related to one another by the partial order given by inclusion. In Section~\ref{subsec: Hess Frob} we apply these general results to the case of local Hessenberg patches to obtain results in characteristic $0$. (An analogous method would also apply to general triangular complete intersections, but we have restricted to Hessenberg patches for our discussion.)

\subsection{Frobenius splittings of  triangular complete intersections}\label{subsec: Frob of triangular lci} 

We begin with a very brief account of the theory of Frobenius splittings, mainly to introduce terminology and establish notation. 
Recall that a commutative ring $R$ is \textbf{reduced} if the map $x\rightarrow x^n$ sends only $0$ to $0$ for any positive integer $n$. 
When $R$  is an $\mathbb{F}_p$-algebra for $p$ a prime, then the $p$-th power map (also called the \textbf{Frobenius map}) is $\mathbb{F}_p$-linear, and if $R$ is reduced, then the kernel of the Frobenius map $x \mapsto x^p$ is equal to $0$, i.e., the Frobenius map is injective. Note that when this injectivity holds, there exists a one-sided \textit{linear} inverse to the Frobenius map; such an inverse can be roughly interpreted as a sort of ``$p$-th root'' map. This motivates the next definition \cite{Brion-Kumar}.

\begin{definition}\label{definition: Frobenius} 
A \textbf{Frobenius splitting} of an $\mathbb{F}_p$-algebra $R$ is a function $\varphi: R\rightarrow R$ which satisfies:

\begin{enumerate}
\item $\varphi(a+b) = \varphi(a)+\varphi(b)$, 
\item $\varphi(a^pb) = a\varphi(b)$, and 
\item $\varphi(1)=1$. 
\end{enumerate}
\end{definition}

\begin{remark} 
Note that by taking $b=1$, conditions (2) and (3) of Definition~\ref{definition: Frobenius} together imply that $\varphi(a^p) = a$. Thus a Frobenius splitting $\varphi$ as in Definition~\ref{definition: Frobenius} is a one-sided inverse to the Frobenius map, as suggested above. Moreover, since $a^p=a$ for $a \in \mathbb{F}_p$, we also see from (1) and (2) that $\varphi$ is $\mathbb{F}_p$-linear. 
\end{remark} 

\begin{remark} 
In fact, Definition~\ref{definition: Frobenius} can be generalized to schemes, but this is not necessary for the purposes of this manuscript. 
\end{remark} 

We also say that an ideal $I\subset R$ is \textbf{compatibly (Frobenius) split} with respect to a given Frobenius splitting $\varphi$ if $\varphi(I)\subset I$. Such ideals have useful properties, detailed more fully in \cite{Knut}. For example,  $I+J$ and $I\cap J$ are compatibly split if $I$ and $J$ are. Note that if $I$ is compatibly split, then the quotient $R/I$ inherits an induced Frobenius splitting $\overline{\varphi}$, so any affine variety defined by a compatibly split ideal is itself Frobenius split (i.e., its coordinate ring is Frobenius split).

We now focus on the case where $R := \mathbb{K}[x_1,\dots,x_n]$ with $\mathbb{K}$ a field of positive characteristic $p$. The \textbf{standard splitting} $\varphi_{std}$ of $R$ is one of the first and most straightforward examples of a Frobenius splitting and is  defined as follows. On a monomial $m$ in $R$, we define 
$$
\varphi_{std}(m) := 
\begin{cases} 
\hfill  \sqrt[p]{m}    \hfill & \text{ if $m$ is a $p$-th power (i.e. there exists $y$ with $y^p=m$)}\\
      \hfill 0 \hfill & \text{ otherwise.}
      \end{cases} 
$$
 The map $\varphi_{std}$ is then extended linearly to all of $R$. It is not difficult to check that the resulting map $\varphi_{std}: R \to R$ satisfies all the conditions of being a Frobenius splitting. (We note that the ideals that are compatibly split by the standard splitting are precisely the Stanley-Reisner ideals \cite[Example 1.1.5]{Brion-Kumar}.)  
Building on this idea, we define the \textbf{trace map} $\tr(\boldsymbol{\cdot})$ as follows. 
First, on a monomial $m$, we define 
\[
\tr(m) :=
  \begin{cases}
      \hfill  \displaystyle\frac{\sqrt[p]{m\prod_{i=1}^n x_i}}{\prod_{i=1}^n x_i}    \hfill & \text{ if $m\prod_{i=1}^n x_i$ is a $p$-th power}\\
      \hfill 0 \hfill & \text{ otherwise}
  \end{cases}
\]
where the product $\prod_{i=1}^n x_i$ appearing above is the product of all the indeterminates in the ring $R$. Second, we define $\tr(\boldsymbol{\cdot}): R \to R$ by extending linearly to all of $R$.   This trace map is not necessarily a Frobenius splitting (however, it is known to be a \textbf{near splitting}, which by definition means that it satisfies the first two conditions of Definition \ref{definition: Frobenius} \cite[Section 1.3.1]{Brion-Kumar}). The trace map can nevertheless be used to build Frobenius splittings, in the following sense: if it is known that if $\tr(f)=1$ for some $f\in R$, then the map 
\begin{equation}\label{eq: def varphi f}
\varphi_f(g) :=\tr(fg)
\end{equation}
defines a Frobenius splitting of $R$ \cite[Section 1.3.1]{Brion-Kumar}. (In fact, it turns out that every Frobenius splitting of $R$ is of this form \cite[Section 1.3.1]{Brion-Kumar} when $\mathbb{K}$ is a perfect field over $\mathbb{F}_p$.) As an example, the reader can easily check that for $f= \prod_{i=1}^n x_i^{p-1}$, then $\varphi_f = \varphi_{std}$ is the standard splitting.

With the above example as motivation and by looking at initial terms, we come to the following lemma.

\begin{lemma}\label{trace_degeneration}
Let $g\in R=\mathbb{K}[x_1,\dots,x_n]$ where $\mathbb{K}$ a field of positive characteristic $p$ and $<$ be a lexicographic monomial order on $R$ such that $\init_<(g)= \prod_{i=1}^nx_i$. Let $f:=g^{p-1}$. Then $\varphi_{f}$ defines a Frobenius splitting of $R$.
\end{lemma}

\begin{proof}
Since $\varphi_f$ is a near-splitting, it suffices to check that $\varphi_f(1) = \tr(f)=1$. 
Since $\init_<(g)=\prod_{i=1}^nx_i$, we must have that $\init_<(f) = \init_<(g^{p-1})=\prod_{i=1}^nx_i^{p-1}$. Therefore, for any other term $m$ of $g^{p-1}$, we must have $m<\prod_{i=1}^nx_i^{p-1}$. But then  $m\prod_{i=1}^nx_i$ is not a $p$-th power, so $\tr(m)=0$. On the other hand $\tr(\prod_{i=1}^nx_i^{p-1})=1$, proving that $\tr(g^{p-1})=\tr(f)=\varphi_f(1)=1$, as required.
\end{proof}

With Lemma~\ref{trace_degeneration} in hand, we now show that triangular complete intersections in $R=\mathbb{K}[x_1,\ldots,x_N]$ can be Frobenius split. Let $I=\langle f_1,\ldots,f_n\rangle \subset \mathbb{K}[x_1,\ldots, x_N]$ be a triangular complete intersection with respect to a monomial order $<$. By a suitable reordering of coordinates, we may assume without loss of generality that $\init_<(f_j)= u_j x_j$ for $1 \leq j \leq n$ and some non-zero constants $u_j \in \mathbb{K}$. Define the polynomial $F_I$ by 
\begin{equation}\label{eq: def FI}
F_I:= \Big{(}\prod_{i>n}x_i\Big{)}\Big{(}\prod_{j\leq n} u_j^{-1} f_j\Big{)} \in \mathbb{K}[x_1,\ldots, x_N].
\end{equation}

It is not difficult to see that $\init_{<}(F_I) = \prod_{i=1}^N x_i$, i.e., $F_I$ is defined so that the initial term of $F_I$ is the product of all the indeterminates in $\mathbb{K}[x_1,\ldots, x_N]$. It follows immediately from Lemma~\ref{trace_degeneration} that we can construct an explicit Frobenius splitting as follows.

\begin{theorem}\label{theorem: FI splitting for triangular}
Let $\mathbb{K}$ a field of positive characteristic $p$. Let $I$ be a triangular complete intersection in $\mathbb{K}[x_1,\ldots,x_N]$ with respect to a monomial order $<$ and assume $\init_{<}(f_j)=u_j x_j$ for $1 \leq j \leq n \leq N$. Let $F_I$ be the polynomial defined by~\eqref{eq: def FI}. Then $\varphi_{F_I^{p-1}} := \tr(F_I^{p-1}\boldsymbol{\cdot})$
is a Frobenius splitting of $\mathbb{K}[x_1,\ldots, x_N]$. 
\end{theorem}

\begin{proof} 
We have just seen that $\init_<(F_I)$ is the product of all the indeterminates appearing in $\mathbb{K}[x_1,\cdots,x_N]$. Setting $g = F_I$ and $f := g^{p-1}=F_I^{p-1}$ in Lemma~\ref{trace_degeneration} yields the claim. 
\end{proof}

From Theorem~\ref{theorem: FI splitting for triangular}, it now follows that for any nonempty subset $L\subseteq [n]$, the ideal $I_L := \langle f_i|i\in L\rangle$ is compatibly split with respect to $\varphi_{F_I^{p-1}}$. We have the following.

\begin{corollary}\label{cor: Hessenberg_split}
Let $\mathbb{K}$ a field of positive characteristic $p$. Let $I$ be a triangular complete intersection in $\mathbb{K}[x_1,\ldots,x_N]$ with respect to a monomial order $<$ and assume $\init_{<}(f_j)=u_j x_j$ for $1 \leq j \leq n \leq N$.  Let $F_I$ be the polynomial defined by~\eqref{eq: def FI}. Let $L$ be a non-empty subset of $[n]$. Then $I_L := \langle f_j \vert  j \in L \rangle$ is a compatibly split ideal with respect to $\varphi_{F_I^{p-1}} = \tr(F_I^{p-1}\boldsymbol{\cdot})$. 
\end{corollary}

\begin{proof}
We use the properties of Frobenius splittings as shown in \cite[Section 1.2]{Brion-Kumar} (or \cite[Section 1]{Knut} for the same statements except without the assumption that $\mathbb{K}$ is algebraically closed). First we claim that $\langle F_I \rangle$ is compatibly split with respect to $\varphi_{F_I^{p-1}}$, which (by definition of compatibly split ideals) is equivalent to $\varphi_{F_I^{p-1}}(\langle F_I \rangle) \subseteq \langle F_I \rangle$. This can be seen by noticing that if $rF_I\in \langle F_I \rangle$, then  $\varphi_{F_I^{p-1}}(rF_I) := \tr(rF_I^p)=F_I\cdot\tr(r)\in \langle F_I\rangle$ where the last equality uses the fact that $\tr(\cdot)$ is a near-splitting (see Definition \ref{definition: Frobenius}).
Next we claim that each principal ideal $\langle f_i \rangle$ is compatibly split with respect to $\varphi_{F_I^{p-1}}$. This is because prime components of a compatibly split ideal are also compatibly split \cite[Proposition 1.2.1]{Brion-Kumar}.  Since $f_i$ is a factor of $F_n$, so is any irreducible factor of $f_i$, and the principal (prime) ideal generated by such an irreducible factor is compatibly split. Thus the intersection of these prime ideals, which is $\langle f_i\rangle$, is also compatibly split, because it is also known that intersections of compatibly split ideals are compatibly split  \cite[Proposition 1.2.1]{Brion-Kumar}.  Finally, it is shown in \cite[Proposition 1.2.1]{Brion-Kumar} that any sum of compatibly split ideals is compatibly split. Since each $I_L$ is obtained by the taking a sum of the principal ideals $\langle f_i \rangle$, the result follows.
\end{proof}

\begin{remark}
It is worth noting that the ideals $I_L$ are examples of Knutson ideals, which are certain collections of ideals that are closed under addition, intersection, and ideal quotients (see \cite[Definition 1]{Sec}). In specific cases, the poset of compatibly split ideals of $\varphi_f$ can be explicitly described using these ideal operations (as in \cite[Theorems 2 and 6]{Knut}). Additionally, it is possible to concatenate Gr\"obner bases of Knutson ideals when $\init_<(f) = \prod_{i=1}^N x_i$ (see \cite[Theorem 1.1]{Sec}).
\end{remark}

\subsection{From positive characteristic to characteristic zero}
\label{subsec: Hess Frob}

As mentioned above, although Definition~\ref{definition: Frobenius} is given in the positive characteristic setting, Frobenius splittings can also provide information about schemes defined over characteristic $0$ (cf. \cite[Section 1.6]{Brion-Kumar} for more details). Our goal in this section is to apply the constructions in Section~\ref{subsec: Frob of triangular lci} to obtain results for Hessenberg varieties. More specifically, in Corollary~\ref{cor: Hessenberg_split} and Corollary~\ref{cor: simultaneous split}, we will give a (local) positive answer to Question 5.1 in the $w_0$-chart. 

We begin by illustrating some of how this works in the study of reducedness. In particular, we recover - using Frobenius techniques - a result from \cite{ADGH} that $\mathrm{Hess}(\mathsf{N},h)$ is reduced in Proposition~\ref{Frobenius_Hess_reduced}. While this result is not new, it demonstrates the utility of these techniques, and allows us to establish some terminology and notation. 

Let $X$ be a separated scheme of finite type over $\spec(\mathbb{Z})$, and let $X_{\mathbb{Q}}$ and $X_{\mathbb{F}_p}$ denote the fibers over $\langle 0\rangle$ and $\langle p\rangle$ respectively. Further, we denote the base change of each to the algebraic closure by $X_{\overline{\mathbb{Q}}}$ and $X_{\overline{\mathbb{F}_p}}$ respectively. Then by \cite[Proposition 1.6.5]{Brion-Kumar}, if $X_{\overline{\mathbb{F}_p}}$ is reduced for all sufficiently large primes $p$, then $X_{\overline{\mathbb{Q}}}$ is also reduced. Since Frobenius split schemes are reduced by \cite[Proposition 1.2.1]{Brion-Kumar}), this hypothesis is implied when $X_{\overline{\mathbb{F}_p}}$ admits a Frobenius splitting for sufficiently large primes $p$.  Finally, recall that if a scheme is reduced over a perfect field -- such as a field of characteristic $0$ -- then it is also reduced over any extension of that field \cite[Section II Exercise 3.15(b)]{Hartshorne}. Thus if $X_{\overline{\mathbb{Q}}}$ is reduced, then so is $X_{\mathbb{C}}$. In other words, the above discussion shows that there is a criterion, phrased in terms of existence of Frobenius splittings in positive characteristic, for a scheme over $\mathbb{C}$ to be reduced. 

In the setting of Hessenberg varieties, since the $f^{w_0}_{k,\ell}$ have integer coefficients, we may consider the scheme over  $\spec(\mathbb{Z})$ defined as $X:=\spec(\mathbb{Z}[\mathbf{x}_{w_0}]/I_{w_0,h}')$ where $I_{w_0,h}'$ is generated by the same $f^{w_0}_{k,\ell}$ as in the previous sections, except that we view them as elements in $\mathbb{Z}[\mathbf{x}_{w_0}]$. By the discussion above, we could prove that  $X_{\mathbb{C}}=\spec(\mathbb{C}[\mathbf{x}_{w_0}]/I_{w_0,h})$ is reduced by showing that there exist appropriate Frobenius splittings for sufficiently large primes. (As mentioned briefly above, there are similar results for other geometric properties of schemes over $\mathbb{C}$, e.g. being Cohen-Macaulay \cite[Proposition 1.6.4]{Brion-Kumar}.) These considerations lead to the following.

\begin{proposition}\label{Frobenius_Hess_reduced}
The Hessenberg local patches $\hess(\mathsf{N},h) \cap \mathcal{N}_{w_0}$ are reduced, and hence $\hess(\mathsf{N},h)$ is reduced. 
\end{proposition} 

\begin{proof} 
It is known that if $Y$ is a \emph{smooth} affine algebraic variety defined over an algebraically closed field of characteristic $p>0$, then $Y$ can be Frobenius split, i.e., there exists a Frobenius splitting $\varphi$ of its affine coordinate ring $A(Y)$ \cite[Proposition 1.1.6]{Brion-Kumar}. 
We next claim that $(\hess(\mathsf{N},h) \cap \mathcal{N}_{w_0})_{\overline{\mathbb{F}_p}}$ is smooth for sufficiently large $p$. Indeed, since the $f_{k,\ell}^{w_0}$ have integer coefficients, the general 
$S$-polynomial reduction will involve at worst rational coefficients. Thus, after clearing denominators and for $p >>0$, no terms in any of the polynomials $f_{k,\ell}^{w_0}$ nor any $S$-polynomials appearing in the Buchberger algorithm would vanish modulo $p$.  In other words, for $p>>0$ the  $\{f_{k,\ell}^{w_0}\}$ (for appropriate $k,\ell$) would remain a Gr\"obner basis of $I_{w_0,h,p}$ with respect to $<_n$. Hence for $p$ sufficiently large, $\init_{<_n}(I_{w_0,h,p})$ is an ideal of indeterminates, and hence the corresponding variety is smooth. Now consider the flat family given by the Gr\"obner degeneration of $I_{w_0,h,p}$ to $\init_{<_n}(I_{w_0,h,p})$. Since being smooth is an open condition in flat families,  if $R/\init_{<_n}(I_{w_0,h,p})$ is regular (i.e. the corresponding variety is smooth), then $R/I_{w_0,h,p}$ is regular too, i.e., $(\hess(\mathsf{N},h) \cap \mathcal{N}_{w_0})_{\overline{\mathbb{F}_p}}$ is smooth for sufficiently large $p$, as desired. From this it follows that for any $p>>0$, there exists a Frobenius splitting of the scheme $(\hess(\mathsf{N},h) \cap \mathcal{N}_{w_0})_{\overline{\mathbb{F}_p}}$, i.e.,  
there exists a Frobenius splitting map $\varphi_p: \overline{\mathbb{F}_p}[\mathbf{x}_{w_0}]/I_{w_0,h,p}
\to \overline{\mathbb{F}_p}[\mathbf{x}_{w_0}]/I_{w_0,h,p}$, where $I_{w_0,h,p}$ is the ideal defined by the same polynomials $f^{w_0}_{k,\ell}$ as in Section~\ref{subsec: reg nilp Hess} but interpreted as elements of 
$\overline{\mathbb{F}_p}[\mathbf{x}_{w_0}]$. Now the argument given in the discussion before the statement of the Proposition yields the claim.  
\end{proof}

The reducedness of $\mathrm{Hess}(\mathsf{N},h)$ is not a new result, as mentioned in Remark \ref{lci}, since it was originally shown in \cite{ADGH}; the point of the above discussion is that it is possible to give an alternative proof using Frobenius splittings. 
A drawback of the considerations so far, however, is that the above considerations yields the \emph{existence} of Frobenius splittings, but we do not obtain concrete information (e.g. a formula) for it. In the setting of local Hessenberg patches, we can remedy this situation using the constructions given in Section~\ref{subsec: Frob of triangular lci}.


Theorem~\ref{theorem: FI splitting for triangular} suggests that one method for explicitly constructing Frobenius splittings is to search for polynomials whose initial term is the product of all the indeterminates of the relevant polynomial ring. This is precisely the strategy that we follow for our Hessenberg patch ideals in $\mathbb{C}[\mathbf{x}_{w_0}]$, which is the general construction from the previous section used in Equation \ref{eq: def FI}. It will be convenient to begin the discussion with the ``largest'' local Hessenberg patch ideal at $w_0$ (in the sense that it contains the largest number of generators $f^{w_0}_{k, \ell}$). This case corresponds to the so-called Peterson Hessenberg function. Specifically, let $n$ be a fixed positive integer with $n \geq 3$.  The Peterson Hessenberg function is defined by $\bar{h} :=(2,3,\ldots,n,n)$, i.e.,  $\bar{h}(i)=i+1$ for $1 \leq i < n$ and $\bar{h}(n)=n$. We consider the ($\overline{\mathbb{F}_p}$-version of the) $w_0$-patch of the Peterson variety, namely $(\hess(\mathsf{N},h) \cap \mathcal{N}_{w_0})_{\overline{\mathbb{F}_p}}$ defined by the ideal $I_{w_0,\bar{h},p} \subset \overline{\mathbb{F}_p}[\mathbf{x}_{w_0}]$ where $p$ is any prime $p>0$. There are $(n-1)(n-2)/2$ many generators $f^{w_0}_{k,\ell}$ of $I_{w_0,\bar{h},p}$, where the indices must satisfy $k > \bar{h}(\ell)=\ell+1$.  By Lemma~\ref{lemma: initial term and one variable}, we know $\init_{<_n}(f^{w_0}_{k,\ell}) = -x_{n+1-k,\ell+1}$. 
Note that the indeterminate $x_{i,1}$ does not appear as the initial term of any $f^{w_0}_{k,\ell}$ for $1\leq i \leq n-1$. With these observations in mind we define the polynomial 

\begin{equation}\label{eq: def Fn}
F_n := (-1)^{(n-1)(n-2)/2}\Big{(}\prod_{1\leq i\leq n-1}x_{i,1}\Big{)}\Big{(}\prod_{k>\bar{h}(\ell)}f_{k,\ell}^{w_0}\Big{)} \in \overline{\mathbb{F}_p}[\mathbf{x}_{w_0}].
\end{equation}

It is not difficult to see that, more or less by construction, $\init_{<_n}(F_n) = \prod\limits_{1\leq j\leq n-i} x_{i,j}$, i.e., the initial term is the product of all the indeterminates in $\overline{\mathbb{F}_p}[\mathbf{x}_{w_0}]$. It follows immediately from Lemma~\ref{trace_degeneration} that we can construct an explicit Frobenius splitting as follows.

\begin{theorem}\label{theorem: Fn splitting for Hess}
Let $p$ be any prime, $p>0$. Let $F_n$ be the function defined by~\eqref{eq: def Fn}. Then $\varphi_{F_n^{p-1}} := \tr(F_n^{p-1}\boldsymbol{\cdot})$
is a Frobenius splitting of $\overline{\mathbb{F}_p}[\mathbf{x}_{w_0}]$. 
\end{theorem}

In fact, by the same arguments in Section~\ref{sec: fij}, the ideals $I_{w_0,h,p}$ are also triangular complete intersections in $\overline{\mathbb{F}_p}[\mathbf{x}_{w_0}]$. From Theorem~\ref{theorem: FI splitting for triangular}, it now readily follows that all Hessenberg patch ideals, for different choices of Hessenberg function $h$, are compatibly split with respect to $\varphi_{F_n^{p-1}}$. We have the following.

\begin{corollary}\label{cor: Hessenberg_split2}
Let $h$ be an indecomposable Hessenberg function for a fixed $n$. Then  the Hessenberg patch ideal $I_{w_0,h,p}$ is a compatibly split ideal with respect to $\varphi_{F_n^{p-1}} = \tr(F_n^{p-1}\boldsymbol{\cdot})$.
\end{corollary}

\begin{proof}
In the notation surrounding the discussion of Corollary~\ref{cor: Hessenberg_split}, observe that each  $I_{w_0,h,p}$ is some $I_L$ for a suitable choice of $L$. An application of Corollary~\ref{cor: Hessenberg_split} to this choice of $L$ yields the result. 
\end{proof}

We have just seen that the Hessenberg patch ideals are compatibly split. In fact, we have just shown that there is a whole family of ideals, related to each other in a natural way, all of which are compatibly split by the same explicit Frobenius splitting $\varphi_{F_n^{p-1}}$ above. Indeed, the following is immediate from the arguments previously given. 

\begin{corollary}\label{cor: simultaneous split}
Let $p>0$ be a prime.  
There is a partially ordered set (ordered by inclusion) of ideals $\{I_{w_0,h,p}\}$, indexed by the set of (indecomposable) Hessenberg functions $h$, which are all compatibly split with respect to the Frobenius splitting $\varphi_{F_n^{p-1}}$.  
\end{corollary}

Corollary~\ref{cor: Hessenberg_split2} and Corollary~\ref{cor: simultaneous split} answer Question~\ref{Frobenius_question} positively in the local patch near $w_0$. It is still an open question whether the same holds in other $w$-charts for $w \neq w_0$ and whether these Frobenius splittings arise as the restriction of a single Frobenius splitting of $\flag(\mathbb{C}^n)$ which simultaneously compatibly splits the regular nilpotent Hessenberg varieties $\hess(\mathsf{N},h)$. We leave this open for future work.

\section{Alternative proof of main results via Liaison}\label{sec: alternative}

In this section, we present an alternative proof of our main results concerning Hessenberg patch ideals (Corollary~\ref{cor: fkl are a GB} and Corollary~\ref{cor: Hess patch is GVD}) using liaison theory. Our main motivation for this section is that we expect these methods to be useful for an analysis of the $w$-charts of $\mathrm{Hess}(\mathsf{N},h)$ for $w \neq w_0$, because the local Hessenberg patch ideals for $w \neq w_0$ are not (necessarily) triangular complete intersections. Hence, the arguments given in the previous sections will not apply, and new ideas will be needed. Moreover, another subtlety arising in analyzing the general case is that the natural grading on the polynomial ring $\mathbb{C}[\mathbf{x}_w]$ for $w \neq w_0$ is not necessarily positive (see Remark~\ref{remark: not positive}). We expect that, in order to handle the general case, the perspectives and tools discussed in this section will be relevant. In particular, we note that in this section, we use liaison theory and an inductive argument (which allows for non-standard gradings) in order to first prove that the ideals are GVD, and from that we deduce that the $f^{w_0}_{k,\ell}$ generators form a Gr\"obner basis. This is in contrast to Section~\ref{sec: background GVD} and Section~\ref{sec: fij}, where we first show that the $f^{w_0}_{k,\ell}$ form a Gr\"obner basis, and then conclude GVD-ness.

 \subsection{Background on liaison theory}\label{subsec: liaison background} 
 
We begin with some background. First we quote a result of Klein and Rajchgot which gives a criterion for an ideal having a geometric vertex decomposition. We say that $I$ is \textbf{square-free} in a variable $y$ if there is a generating set $\mathcal{G}$ of $I$ such that $y^2$ does not divide any term of any element of $\mathcal{G}$. Note that, in the statement of the theorem below, there is no requirement for homogeneity. We comment on this further, below. 

\begin{theorem}\label{theorem: Klein Rajchgot}(\cite[Theorem 6.1]{KleRaj})
Let $I$, $C$, and $N \subseteq I \cap C$ be ideals of $R=\CC[x_1,\ldots,x_n]$, and let $<$ be a $y$-compatible term order for some $y\in\{x_1,\ldots,x_n\}$.  Suppose that $I$ is square-free in $y$ and that no term of any element of the reduced Gr\"obner basis of $N$ is divisible by $y$.  Suppose further that there exists an isomorphism $\varphi: C/N \xrightarrow{f/g} I/N$ of $R/N$-modules for some $ f, g \in R$ not zero-divisors in $R/N$, and $\text{in}_y(f)/g = y$.  Then $\text{in}_y I = C \cap ( N+\langle y \rangle )$ is a geometric vertex decomposition of $I$.  
\end{theorem}

As mentioned above, in this section we first find geometric vertex decompositions, and then use them to find Gr\"obner bases.  The idea is that if we are in the setting of Theorem~\ref{theorem: Klein Rajchgot}, then we can obtain a Gr\"obner basis for $I$ by using Gr\"obner bases for $C$ and $N$. Lemma~\ref{lemma: GVD to GB} below makes this idea precise. We note that this result is essentially already contained in the proof of \cite[Corollary 4.13]{KleRaj}, but we chose to state it explicitly in this form for the following reason. In the statement of \cite[Corollary 4.13]{KleRaj}, Klein and Rajchgot give alternate criteria which guarantees the existence of an $R/N$-module isomorphism $\varphi: C/N \to I/N$ (which is a necessary hypothesis in Theorem~\ref{theorem: Klein Rajchgot} above). However, in our arguments below, we have other explicit methods to prove the existence of the necessary isomorphisms, so we restate the result in a form best suited for our purposes. 

Another preliminary remark is in order. The proof of \cite[Corollary 4.13]{KleRaj} (and hence Lemma~\ref{lemma: GVD to GB} below) relies on \cite[Lemma 4.12]{KleRaj}, which is itself a restatement of a result from liaison theory \cite[Lemma 1.12]{GorlaMN}. The original result in \cite{GorlaMN}, commonly called constructing a Gr\"obner basis via linkage, is phrased in terms of liaison-theoretic constructions. To avoid the technical liaison-theoretic setup, we instead opt for the graded isomorphism phrasing found in Lemma~\ref{non_standard_grading} below \cite{Neye}, which additionally allows for \emph{non-standard} gradings. This last point is important for us, since (as we have seen in Section~\ref{subsec: nonstandard}) our rings have non-standard gradings. 

Recall that a $\mathbb{Z}^d$-grading on a polynomial ring $R$ is said to be \emph{positive}, or equivalently we say that the polynomial ring $R$ is \emph{positively graded}, if the 
only elements in $R$ of degree $0$ are the constants.

\begin{lemma}\label{non_standard_grading}(\cite[Lemma 3.4]{Neye})
Let $R$ be a positively $\mathbb{Z}^d$-graded polynomial ring over an arbitrary field $\mathbb{K}$. Let $I,C$ and $N$ be homogeneous ideals with respect to the given $\mathbb{Z}^d$-grading, such that $N \subseteq I \cap C$. Let $A$ be a monomial ideal of $R$ such that $A \subseteq \init_{<}(I)$ and $\init_{<}(N) \subseteq A$ for some monomial order $<$. 
Suppose that there exists $e \in \mathbb{Z}^d$ such that $(I/N)_{\ell}\cong (C/N)_{\ell-e}$ and $(A/\init_{<}(N))_{\ell} \cong (\init_{<}(C)/\init_{<}(N))_{\ell-e}$ as $\mathbb{K}$-vector spaces for all $\ell \in \mathbb{Z}^d$. Then $A = \init_{<}(I)$. 
\end{lemma}

We are almost ready to state and prove Lemma~\ref{lemma: GVD to GB}, but there are two final things to note. Firstly, in the situation in which we wish to apply Lemma~\ref{non_standard_grading}, the monomial ideal $A$ will be an ideal $\tilde{I}$ generated by the initial terms of a proposed Gr\"obner basis. Applying Lemma~\ref{non_standard_grading} will then allow us to conclude that the proposed basis is in fact Gr\"obner. Secondly, to apply Lemma~\ref{non_standard_grading}, we can see from its hypotheses that two separate isomorphisms are needed. The isomorphism between the graded pieces of $C/N$ and $I/N$ arises from the isomorphism $\varphi$ mentioned in Theorem~\ref{theorem: Klein Rajchgot}, and while Theorem~\ref{theorem: Klein Rajchgot} does not require that $\varphi$ be graded or that $I,C$ and $N$ be homogeneous, Lemma~\ref{non_standard_grading} does; in our case, $\varphi$ does respect the appropriate grading. The isomorphisms between the (shifted) graded pieces of $\init_<(C)/\init_<(N)$ and $\tilde{I}/\init_<(N)$, on the other hand,  must be explicitly constructed, and this occupies much of our proof below.

\begin{lemma}\label{lemma: GVD to GB}
Let $R = \mathbb{C}[x_1,\ldots,x_n]$ be a positively $\mathbb{Z}^d$-graded polynomial ring over $\mathbb{C}$. 
Let $I,C$ and $N$ be homogeneous ideals with respect to the given $\mathbb{Z}^d$-grading, such that $N \subseteq I \cap C$. Let $<$ be a $y$-compatible term order for some $y \in \{x_1,\ldots,x_n\}$ and assume that $y$ is a homogeneous element in $R$. Suppose further that $I$ is square-free in $y$ and that no term of any element of the reduced Gr\"obner basis of $N$ is divisible by $y$. Also assume that there exists an isomorphism $\varphi: C/N \stackrel{f/g}{\rightarrow} I/N$ of $R/N$-modules for some $f, g \in R$ not zero-divisors in $R/N$, where $in_y(f)/g = y$, and which shifts degrees by $\deg(y)$. In the notation of Theorem~\ref{theorem: Klein Rajchgot}, suppose that $\{q_1,\ldots,q_k,h_1,\ldots,h_\ell\}$ and $\{h_1,\ldots,h_\ell\}$ are Gr\"obner bases for $C$ and $N$ respectively, with respect to the $y$-compatible monomial order $<$. Suppose $r_i$ for $1 \leq i \leq k$ are polynomials in $R$ which do not contain any $y$'s, and such that $y q_i + r_i \in I$.   Then $\{yq_1+r_1,\ldots, yq_k+r_k, h_1,\ldots, h_\ell\}$ is a Gr\"obner basis for $I$ with respect to $<$. 
\end{lemma} 

\begin{proof} 
Let $\tilde{I} := \langle \init_<(yq_1+r_1),\ldots, \init_<(yq_k+r_k), \init_<(h_1),\ldots ,\init_<(h_{\ell}) \rangle \subseteq \init_<(I)$. To prove the desired conclusion of the lemma, it would suffice to show $\tilde{I} = \init_<(I)$. By assumption, we know $$\init_<(C) = \langle \init_<(q_1),\ldots,\init_<(q_k),\init_<(h_1),\ldots, \init_<(h_{\ell})\rangle$$ and $$\init_<(N) = \langle \init_<(h_1),\ldots, \init_<(h_{\ell})\rangle.$$
Since $<$ is a $y$-compatible monomial order, we have $\init_<(yq_i+r_i) = y\cdot\init_<(q_i)$ for $1\leq i\leq k$. Since 
$\init_<(C)/\init_<(N)$ is generated by $\langle \init_<(q_1), \ldots,\init_<(q_k)\rangle$ (where we slightly abuse notation and use the same symbols to denote elements in $\init_<(C)$ and their equivalence classes in the quotient) and similarly $\tilde{I}/\init_<(N)$ is generated by $\langle \init_<(yq_1+r_1), \ldots, \init_<(yq_k+r_k) \rangle$, the equality $\init_<(yq_i+r_i) = y\cdot\init_<(q_i)$ implies that the graded $R/\init_<(N)$-module map  $[\init_<(C)/\init_<(N)](-\mathrm{deg}(y)) \rightarrow \tilde{I}/\init_<(N)$ defined by multiplication by $y$ is an isomorphism, where $\deg(y) \in \mathbb{Z}^d$ is the degree of $y$ in the given $\mathbb{Z}^d$-grading. 

By assumption, we have the necessary isomorphisms of graded pieces of $C/N$ and $I/N$, so we may now apply Lemma \ref{non_standard_grading} with $A = \tilde{I}$ and $e = \mathrm{deg}(y)$ to conclude that $\tilde{I} = \init_<(I)$. This proves that $\{yq_1+r_1,\ldots,yq_k+r_k, h_1,\ldots,h_\ell\}$  is a Gr\"obner basis for $I$ with respect to $<$, as was to be shown. 
\end{proof}

We will use Theorem~\ref{theorem: Klein Rajchgot} and Lemma~\ref{lemma: GVD to GB} in the arguments below in an inductive process. More precisely, in order to prove that a certain set of generators for our Hessenberg patch ideals is a Gr\"obner basis, we will build a sequence of choices of $y, C, N$ etc., where at each stage we can show the relevant isomorphism, thus enabling us to apply the above results inductively to prove that our generating set is Gr\"obner.

 Finally, we record here -- for future use -- a version of a result of Klein and Rajchgot for the case of non-standard gradings. In fact, we could have used the result below in order to prove the results in the later sections, but we chose to use Lemma~\ref{lemma: GVD to GB} instead. 

\begin{proposition}
Let $R$ be a positively $\mathbb{Z}^d$-graded polynomial ring over an arbitrary field $\mathbb{K}$. Let $I = \langle yq_1+r_1,\dots, yq_k+r_k,h_1,\dots, h_\ell \rangle$ be a homogeneous ideal of $R$ with respect to the given $\mathbb{Z}^d$-grading, with $y = x_j$ some variable of $R$ and $y$ not dividing any term of any $q_i$ for $1 \leq i \leq k$ nor of any $h_j$ for $1 \leq j \leq \ell$. Fix a term order $<$, and suppose that $\mathcal{G}_C = \{q_1,\dots, q_k,h_1,\dots, h_\ell\}$ and $\mathcal{G}_N = \{h_1,\dots, h_\ell\}$ are Gr\"obner bases for the ideals they generate, which we call $C$ and $N$, respectively.  Assume that $\text{in}_<(yq_i+r_i) = y \cdot \text{in}_<(q_i)$ for all $1 \leq i \leq k$. Assume also that $\hgt(I)$, $\hgt(C)>\hgt(N)$ and that $N$ is unmixed.  Let $M = \begin{pmatrix}
q_1& \cdots & q_k\\
r_1& \cdots & r_k
\end{pmatrix}.$ If the ideal of $2$-minors of $M$ is contained in $N$, then the given generators of $I$ are a Gr\"obner basis.
\end{proposition}
\begin{proof}
Follow the proof of \cite[Corollary 4.13]{KleRaj} line-by-line, using Lemma \ref{non_standard_grading} in place of \cite[Lemma 4.12]{KleRaj} and changing the degree shifts from $1$ to $\deg(y)$.
\end{proof}

\subsection{Induction for regular nilpotent Hessenberg varieties}\label{subsec: induction liaison}

As we mentioned above, our liaison-theoretic argument proceeds by an induction on $n$. Before the generalities, we present an example which will illustrate the idea.

\begin{example}\label{example: inductive ideals} 
Let $n=5$ and $h=(2,3,4,5,5)$. 
Recall that we visualize the polynomials $f^{w_0}_{k,\ell}$ as the $(k,\ell)$-th matrix entries in a $5 \times 5$ matrix as follows:
\begin{equation*}
\begin{bmatrix}
0 & 0 & 0 & 0 & 0\\
1 & 0 & 0 & 0 & 0\\
f^{w_0}_{3,1} & 1 & 0 & 0 & 0\\
f^{w_0}_{4,1} & f^{w_0}_{4,2} & 1  &  0 &  0 \\
f^{w_0}_{5,1} & f^{w_0}_{5,2} & f^{w_0}_{5,3} & 1 & 0 \\
\end{bmatrix}
\end{equation*}
Since $h=(2,3,4,5,5)$ and the ideal $I_{w_0,h}$ is generated by the polynomials $f^{w_0}_{k,\ell}$ with $k>h(\ell)$, we have 
$$
I_{w_0,h} = \langle f^{w_0}_{3,1}, f^{w_0}_{4,1}, f^{w_0}_{4,2}, f^{w_0}_{5,1}, f^{w_0}_{5,2}, f^{w_0}_{5,3} \rangle.
$$ 
Now we can visualize a sequence of ideals $I_{w_0,h}(m)$ by considering a sequence of matrices with ``crossed out'' entries as follows
\begin{equation*}
\resizebox{.95\textwidth}{!}{    
$\begin{bmatrix}
0 & 0 & 0 & 0 & 0\\
1 & 0 & 0 & 0 & 0\\
f^{w_0}_{3,1} & 1 & 0 & 0 & 0\\
f^{w_0}_{4,1} & f^{w_0}_{4,2} & 1  &  0 &  0 \\
\xcancel{f^{w_0}_{5,1}} & \xcancel{f^{w_0}_{5,2}} & \xcancel{f^{w_0}_{5,3}} & 1 & 0 \\
\end{bmatrix}
\longrightarrow
\begin{bmatrix}
0 & 0 & 0 & 0 & 0\\
1 & 0 & 0 & 0 & 0\\
f^{w_0}_{3,1} & 1 & 0 & 0 & 0\\
f^{w_0}_{4,1} & f^{w_0}_{4,2} & 1  &  0 &  0 \\
\xcancel{f^{w_0}_{5,1}} & \xcancel{f^{w_0}_{5,2}} & f^{w_0}_{5,3} & 1 & 0 \\
\end{bmatrix}
\longrightarrow
\begin{bmatrix}
0 & 0 & 0 & 0 & 0\\
1 & 0 & 0 & 0 & 0\\
f^{w_0}_{3,1} & 1 & 0 & 0 & 0\\
f^{w_0}_{4,1} & f^{w_0}_{4,2} & 1  &  0 &  0 \\
\xcancel{f^{w_0}_{5,1}} & f^{w_0}_{5,2} & f^{w_0}_{5,3} & 1 & 0 \\
\end{bmatrix}
\longrightarrow
\begin{bmatrix}
0 & 0 & 0 & 0 & 0\\
1 & 0 & 0 & 0 & 0\\
f^{w_0}_{3,1} & 1 & 0 & 0 & 0\\
f^{w_0}_{4,1} & f^{w_0}_{4,2} & 1  &  0 &  0 \\
f^{w_0}_{5,1} & f^{w_0}_{5,2} & f^{w_0}_{5,3} & 1 & 0 \\
\end{bmatrix}$
}
\end{equation*}
We can now define the ideals $I_{w_0,h}(m)$ where the integer $m$ indicates the number of crossed-out entries, and $I_{w_0,h}(m)$ is generated by the subset of generators of $I_{w_0,h}$  which are \emph{not} crossed out. Thus for instance the left-most matrix above, with $3$ crossed-out entries, corresponds to the ideal $I_{w_0,h}(3)$ generated by $f^{w_0}_{3,1}, f^{w_0}_{4,1}$ and $f^{w_0}_{4,2}$. 
It is evident from this description that the $4$ matrices above corresponds to an increasing sequence of ideals
\[
I_{w_0,h}(3) \subset I_{w_0,h}(2) \subset I_{w_0,h}(1) \subset I_{w_0,h}(0)\]
and that $I_{w_0,h}(0)=I_{w_0,h}$.

\end{example}

We now formalize the construction given in Example~\ref{example: inductive ideals}. Suppose that $n \in {\mathbb{Z}}$ and $n \geq 3$. Let $h:[n]\rightarrow [n]$ be an indecomposable Hessenberg function. To avoid the trivial case in which the Hessenberg variety is equal to the whole flag variety, we additionally assume throughout this section that $h \neq (n,n,\ldots,n)$. In particular, there exists some $L \in [n]$ such that $h(L) <n$. Then
\begin{equation}\label{eq: def mu}
\mu(h):=\max\{L \, \mid \, h(L)<n\}
\end{equation}
is well-defined. 
Suppose now that $m$ is an integer such that $0\leq m \leq \mu(h)$. 
Next we define 
\begin{equation}\label{eq: Hhk}
H(h,m) := \{(k,\ell)\in [n]\times [n] \, \mid \,  h(\ell)<k<n\} \sqcup \{(k,\ell)\in [n]\times [n]: k=n, h(\ell)<n, \ell>m\}
\end{equation}
and also define 
\begin{equation}\label{eq: Iw0hk}
I_{w_0,h}(m) := 
\langle f^{w_0}_{k,\ell} \, \mid \, (k,\ell) \in H(h,m) \rangle.
\end{equation}
Notice that if 
$m < \mu(h)$ then 
there is at least one pair $(k,\ell)$ which is contained in the 
second set described in the RHS of~\eqref{eq: Hhk}, whereas if $m=\mu(h)$ then 
the second set is empty. 
Put another way, if $m<\mu(h)$ then $I_{w_0,h}(m)$ contains at least one polynomial $f^{w_0}_{k,\ell}$ for which the first index $k$ is equal to $n$, while there is no such generator for $I_{w_0,h}(m=\mu(h))$. Furthermore, it is easy to see that $I_{w_0,h} = I_{w_0,h}(0)$.

\begin{example} 
Continuing in the $n=5$ setting of Example~\ref{example: inductive ideals}, we see that $\mu(h)=3$ in that case. Moreover, as observed above, the ideal $I_{w_0,h}(3)$ contains no generators of the form $f^{w_0}_{n,\ell} = f^{w_0}_{5,\ell}$ for any $\ell$, because all such generators have been ``crossed off''. It is also easy to see that $I_{w_0,h}(0),I_{w_0,h}(1),I_{w_0,h}(2)$ still contain generators of the form $f^{w_0}_{5,\ell}$. 
\end{example}

As we just saw, the generators $f^{w_0}_{n,\ell}$ do not appear in $I_{w_0,h}(m)$ when $m=\mu(h)$. This allows us to make an inductive argument connecting $I_{w_0,h}(m)$ with an analogous ideal for the $n-1$ case, and it is the recursive structure of these ideals which allows us to prove our main results. 
We make this more precise in the next lemma.

\begin{lemma}\label{lemma: relabeling}
Suppose that $m=\mu(h)$ and $n>3$. Let  $\bar{h}:[n-1]\rightarrow [n-1]$ be the Hessenberg function on $[n-1]$ defined by $\bar{h}(\ell) = h(\ell)$ if $h(\ell)<n$ and $\bar{h}(\ell) =n-1$ otherwise. Denote the longest element of $S_{n-1}$ by $\overline{w_0}$. Then 
\begin{enumerate} 
\item the injective ring homomorphism 
$\varphi_{n-1,n}: \CC [\mathbf{x}_{\overline{w_0}}]\rightarrow \CC [\mathbf{x}_{w_0}]$ which sends $x_{i,j}$ to $x_{i+1,j}$ satisfies $ \varphi_{n-1,n}(f^{\overline{w_0}}_{k,\ell}) = f^{w_0}_{k,\ell}$  and 
\item the generators of $I_{w_0,h}(m)$ lie in the image of $\varphi_{n-1,n}$ and the ideal generated in $\CC[\mathbf{x}_{\overline{w_0}}]$ by their (unique) preimages under $\varphi_{n-1,n}$ is the ideal $I_{\overline{w_0}, \bar{h}}$ corresponding to the smaller Hessenberg function $\bar{h}$. 
\end{enumerate} 
\end{lemma} 

\begin{proof} 
Claim (1) is straightforward to see from the explicit descriptions of the generators $f^{w_0}_{k,\ell}$ given in Section~\ref{subsec: reg nilp Hess}. Claim (2) then follows from Claim (1) from the definition of $\varphi_{n-1,n}$ and the assumption that $m=\mu(h)$. 
\end{proof} 

The identification of the generators
for $I_{w_0,h}(m) \subseteq \CC [\mathbf{x}_{w_0}]$ with those of 
$I_{\overline{w_0}, \bar{h}} \subseteq \CC [\mathbf{x}_{\overline{w_0}}]$ as described in Lemma~\ref{lemma: relabeling} will be a key component of our arguments. We give a simple example to illustrate the idea. 

\begin{example}\label{injection_example}
For the purpose of this example, let $w_0$ denote the longest element in $S_5$ and $\overline{w_0}$ denote the longest element in $S_4$. We can compare the polynomials  $f_{k,\ell}^{\overline{w_0}}$ and $f_{k,\ell}^{w_0}$ explicitly in this case. For $n=4$ we can compute that $f^{\overline{w_0}}_{3,1} = x_{3,1} - x_{2,2}, f^{\overline{w_0}}_{4,1} = x_{2,1}  - x_{1,2}- x_{1,3}(x_{3,1}-x_{2,2})$ and $f^{\overline{w_0}}_{4,2} = x_{2,2} - x_{1,3}$, whereas for $n=5$ we have $f^{w_0}_{3,1} = x_{4,1}-x_{3,2}, f^{w_0}_{4,1} = x_{3,1} - x_{2,2} - x_{2,3}(x_{4,1} - x_{3,2})$ and 
$f^{w_0}_{4,2} = x_{3,2} - x_{2,3}$. This illustrates the claim of the above lemma that if  the variables $x_{i,j}$ for the $n=4$ case get sent to $x_{i+1,j}$ in the $n=5$ case then the polynomials $f^{\overline{w_0}}_{k,\ell}$ map to $f^{w_0}_{k,\ell}$. 
\end{example}

We also need the following result of the first author, Cummings, Rajchgot, and Van Tuyl \cite{CDSRVT}. 

\begin{theorem}\label{theorem: extension GVDs}(\cite[Theorem 2.9]{CDSRVT})
Let $I \subsetneq R=k[x_1,\ldots,x_n]$ and $J \subsetneq S = k[y_1,\ldots,y_m]$ be two proper ideals. Then $I$ and $J$ are geometrically vertex decomposable if and only if $I+J$ is geometrically vertex decomposable in $R \otimes_k S = k[x_1,\ldots, x_n, y_1, \ldots, y_m]$. 
\end{theorem} 

We can now state and prove the first main result of this section.

\begin{theorem}\label{GVD_main}
Let $n$ be a positive integer with $n \geq 3$.  Let $h: [n] \to [n]$ be an indecomposable Hessenberg function, and let $0\leq m\leq \mu(h)$. Then the ideal $I_{w_0,h}(m)$ is geometrically vertex decomposable.
\end{theorem}

\begin{proof}
Let $r := \lvert H(h,m) \rvert$, i.e., the number of generators $f^{w_0}_{k,\ell}$ defining $I_{w_0,h}(m)$.
We will prove the result using a double induction argument on $n\geq 3$ and on $r :=|H(h,m)|$. 

First, we show that the claim of the theorem is true for $r=1$ and any $n\geq 3$. To do this, we must check the conditions of Definition~\ref{def: g v decomposable}. We first address the unmixedness condition. In the case $r=1$, the ideal $I_{w_0,h}(m)$ is principal, generated by a single element $f^{w_0}_{k,\ell}$.  Being a principal ideal, it is immediate in this case that $I_{w_0,h}(m)$ is a complete intersection and hence unmixed (see \cite[Proposition 18.13 \& Corollary 18.14]{Eisenbud}). Next, we need to show that $I_{w_0,h}(m)$ satisfies either condition (1) or (2) of Definition~\ref{def: g v decomposable}. Since $I_{w_0,h}(m)$ is neither $\langle 1 \rangle$ nor generated by indeterminates, we must show that it satisfies condition (2). To do this, note that the only way that $r=1$ can occur is if either $m=\mu(h)-1$ and the unique generator is $f^{w_0}_{n, \mu(h)}$, or, $m=\mu(h)$ and the unique generator is $f^{w_0}_{n-1,1}$. In either case, we note that the ideal being principal implies that the generator also forms a Gr\"obner basis with respect to $<_n$. 

We take cases. Suppose $m=\mu(h)-1$ and the unique generator is $f^{w_0}_{n, \mu(h)}$. By Lemma~\ref{lemma: initial term and one variable} we know $\mathrm{in}_{<_n}(f^{w_0}_{n,\mu(h)}) = - x_{1, \mu(h)+1}$. (Note that $\mu(h)<n-1$ since we assume $h$ is indecomposable. Hence $\mu(h)+1<n$ and thus $x_{1,\mu(h)+1}$ is a valid variable in $\mathbf{x}_{w_0}$.) 
This implies that if we choose $y=x_{1,\mu(h)+1}$ in the construction outlined in Section~\ref{sec: background GVD}, then the corresponding ideals are $C_{y,I_{w_0,h}(m)} =\langle 1\rangle$ and $N_{y,I_{w_0,h}(m)} =\langle 0\rangle$, both of which are geometrically vertex decomposable. 
Moreover, $\init_{<_n}(I_{w_0,h}(m)) = \langle x_{1,\mu(h)+1} \rangle
= C_{y, I_{w_0,h}(m)} \cap (N_{y, I_{w_0, h}(m)} + \langle y=x_{1,\mu(h)+1}\rangle)$, so we obtain a geometric vertex decomposition by Definition~\ref{def: g v decomposable}. (Note that since $f^{w_0}_{n,m+1}$ is square-free in $y$, this also follows from \cite[Theorem 2.1]{KMY}). Now we take the other case; suppose the unique generator of $I_{w_0,h}(m)$ is $f^{w_0}_{n-1,1}$. Note that by assumption on indecomposability, this case can only occur if $n-1=k>\ell+1=2$, i.e., $n>3$. Now by Lemma~\ref{lemma: initial term and one variable} we know $\mathrm{in}_{<_n}(f^{w_0}_{n-1,1}) = x_{2,2}$ which is a valid variable in $\mathbf{x}_{w_0}$ since $n>3$. Choosing $y=x_{2,2}$ and proceeding with the argument as in the previous case yields the desired claim. This concludes the proof for the cases in which $r=1$, for any $n \geq 3$.

We now proceed with the inductive argument. Let $r \geq 2$ and fix an $n \geq 3$. We assume that the claim holds for any $n \geq 3$ and for $r-1$. Suppose $h$ and $m$ are such that $\lvert H(h,m) \rvert = r$ and consider the ideal $I_{w_0,h}(m)$. Unlike the base cases, we do not a priori have a Gr\"obner basis for $I_{w_0,h}(m)$, so checking the conditions of a geometric vertex decomposition using the Gr\"obner basis construction of Definition~\ref{def: GVD} is not as immediate. Thus, instead of working directly with Gr\"obner bases, we will use the result of Klein and Rajchgot recorded in Theorem~\ref{theorem: Klein Rajchgot} to construct a geometric vertex decomposition.

We consider two cases. 
Suppose first that $m\leq\mu(h)-1$ and set $I=I_{w_0,h}(m)$.  We define $C := \langle 1\rangle$ and $N := I_{w_0,h}(m+1)$. Clearly $N\subset C\cap I_{w_0,h}(m)$ as $C=R=\CC [\mathbf{x}_{w_0}]$.  Since $h$ is indecomposable, we have $\mu(h) \leq n-2$,
and since $m \leq \mu(h)-1$ by assumption we have $m+2 < n$. Hence $x_{1,m+2}\in\CC[\mathbf{x}_{w_0}]$ and we can set $y := x_{1,m+2}$.  We know from its definition that $I_{w_0,h}(m)$ contains a generator of the form $f^{w_0}_{n,m+1}$ where $f^{w_0}_{n,m+1} \neq 1$ (since $n> (m+1)+1=m+2$). In this situation, by Lemma~\ref{lemma: initial term and one variable}, $I_{w_0,h}(m)$ is square-free in $y=x_{1,m+2}$, and the generator $f^{w_0}_{n,m+1}$ is the only element in the set of generators $f^{w_0}_{k,\ell}$ of $I_{w_0,h}(m)$ in which the variable $y$ appears. Therefore, it follows from the standard constructions of Gr\"obner bases that no term of any element of the reduced Gr\"obner basis of $N=I_{w_0,h}(m+1)$ with respect to $<_n$ is divisible by $y$.   Moreover, we can see that $I/N$ is a rank one $R/N$-module generated by $f^{w_0}_{n,m+1}$, and it is a free module since $f^{w_0}_{n,m+1}$ contains a $y$ whereas no generator in $N$ contains a $y$ (ie. $f^{w_0}_{n,m+1}$ is not a zero-divisor in $R/N$). We can also see that $C/N$ is the rank one $R/N$-module generated by $1$. 
It follows that multiplication by $-f^{w_0}_{n,m+1}$ defines an isomorphism $R/N \cong C/N\rightarrow I/N$, where we know that $\init_{y}(-f^{w_0}_{n,m+1}) = y$. Applying Theorem~\ref{theorem: Klein Rajchgot}, we conclude that these choices of $y, C, N$ define a geometric vertex decomposition $\init_y(I_{w_0, h}(m)) = C \cap (N + \langle y \rangle)$ of $I_{w_0,h}(m)$. Now note that $N$ corresponds to an ideal with $r-1$ generators, so by induction on $r$, $N$ is geometrically vertex decomposable, and $C=\langle 1 \rangle$ is geometrically vertex decomposable by definition. To complete the proof that $I_{w_0,h}(m)$ is geometrically vertex decomposable, the only thing that remains to prove is that $I_{w_0,h}(m)$ is unmixed. 

To see that $I_{w_0,h}(m)$ is unmixed, observe that $\init_y(I_{w_0, h}(m)) = C \cap (N + \langle y \rangle)$ is a degenerate geometric vertex decomposition since $C=\langle 1\rangle$.
By the discussion in \cite{KleRaj} before \cite[Proposition 2.4]{KleRaj}, this implies that there is a unique element in the reduced Gr\"obner basis of $I$ of the form $uy+g$ where $u$ is a unit and $g$ does not contain $y$. In particular this means that $y$ can be written in terms of the other variables, and thus $R/I_{w_0,h}(m)$ is isomorphic to $R/(\langle y \rangle + N)$. Now observe that since $N$ is geometrically vertex decomposable by induction, it must be unmixed. Furthermore, $y$ does not divide any term of its reduced Gr\"obner basis. Therefore, if $\cap_iP_i$ is a primary decomposition of $N$, then $\cap_i(P_i+\langle y\rangle)$ is a primary decomposition of $N+\langle y \rangle$. But in this case the dimension conditions for unmixedness remain true, so $N+\langle y\rangle$ is also unmixed. Therefore, $I_{w_0,h}(m)$ is unmixed as well.

 We take a moment to note that, since $h$ is indecomposable, the pairs $(k,\ell)$ that correspond to potential generators $f_{k,\ell}^{w_0}$ for the ideals $I_{w_0,h}(m)$ for any value of $m$ must satisfy $k>\ell+1$. Hence for a given value of $n$ (with $n \geq 3$), the values of $r$ that can occur -- for any value of $m$ -- have an a priori upper bound of $n(n-1)/2$. Therefore, we may proceed by showing that the claim of the theorem holds for any allowed value of $r$ for both $n=3$ and $n=4$, and then induct on both the value of $n$ and on $r$. In particular, in the argument that follows, we can assume that the claim is true for $n-1$ and for any allowed value of $m$ for $n-1$. 

Let us now consider the remaining case, when $m=\mu(h)$. In this case, 
the ideal $I_{w_0,h}(m)=I_{w_0,h}(\mu(h))$ does not contain any generators of the form $f^{w_0}_{n,\ell}$ for any $\ell$. By Lemma~\ref{lemma: relabeling}, we know that the generators of $I_{w_0,h}(m=\mu(h))$ are precisely the images under the map $\varphi_{n-1,n}: \CC[\mathbf{x}_{\overline{w_0}}] \to \CC[\mathbf{x}_{w_0}]$ of the analogous generators $f^{\overline{w_0}}_{k,\ell}$ of the ideal $I_{\overline{w_0}, \bar{h}}$. Note that $I_{\overline{w_0}, \bar{h}}$ is a special case of an ideal of the form $I_{\overline{w_0}, \bar{h}}(m)$, and it is associated to a smaller value of $n$, since $\bar{h}$ is a Hessenberg function on $[n-1]$, not $[n]$. Thus, by the induction hypothesis on $n$, we may assume that $I_{\overline{w_0}, \bar{h}}$ is geometrically vertex decomposable in $\CC[\mathbf{x}_{\overline{w_0}}]$.
Now, by applying Theorem~\ref{theorem: extension GVDs} to the case $I=I_{\overline{w_0}, \bar{h}}$ and $J=0$, where $R=\CC[\mathbf{x}_{\overline{w_0}}]$ and $S$ is the polynomial ring generated by $\mathbf{x}_{w_0} \setminus \varphi_{n-1,n}(\mathbf{x}_{\overline{w_0}})$,  we may conclude that the ideal $I_{w_0,h}(\mu(h))$ is also geometrically vertex decomposable in $\CC[\mathbf{x}_{w_0}]$. 
This completes the induction step and hence the proof. 
\end{proof}

Since the ideals $I_{w_0,h}$ are special cases of the ideals $I_{w_0,h}(m)$, the following is immediate.

\begin{corollary}\label{corollary: GVD_main}
Let $n$ be a positive integer with $n \geq 3$. Let $h: [n] \to [n]$ be an indecomposable Hessenberg function. 
Then the Hessenberg patch ideal $I_{w_0,h}$ of $\hess(\mathsf{N},h)$ in the $w_0$-chart is geometrically vertex decomposable.
\end{corollary}

\begin{example}\label{example: demo module iso}

Continuing with Example~\ref{n5_example}, we take $n=5, h=(2,3,4,5,5)$ and $w_0=[5 \, 4 \, 3 \, 2 \, 1]$. Set $R=\CC[\mathbf{x}_{w_0}]$. We outline in more detail how our proof above shows that $I_{w_0,h}=I_{w_0,h}(0)$ is geometrically vertex decomposable, assuming that we have shown geometric vertex decomposability for the $n=4$ cases.
We begin the explanation by constructing a geometric vertex decomposition for $I_{w_0,h}(2)$. The matrix  $(w_0M)^{-1}\mathsf{N}(w_0M)$ 
is of the form 
\[
  \begin{bmatrix}
    0 & 0 & 0 & 0 & 0\\
    1 & 0 & 0 & 0 & 0\\
    f^{w_0}_{3,1} & 1 & 0 & 0 & 0\\
    f^{w_0}_{4,1} & f^{w_0}_{4,2} & 1  &  0 &  0 \\
    * & * & f^{w_0}_{5,3} & 1 & 0 \\
  \end{bmatrix}
\]
\noindent where
\begin{align*}
    & f^{w_0}_{5,3} = -x_{1,4}+x_{2,3}\\
    & f^{w_0}_{4,1} = -x_{2,2}+x_{2,3}(x_{3,2}-x_{4,1})+x_{3,1}\\
    & f^{w_0}_{4,2} =-x_{2,3}+x_{3,2}\\
    & f^{w_0}_{3,1} =-x_{3,2}+x_{4,1}. 
\end{align*}
and by definition we have $I_{w_0,h}(2) =\langle f^{w_0}_{3,1},f^{w_0}_{4,1},f^{w_0}_{4,2},f^{w_0}_{5,3}\rangle$. Following the inductive procedure suggested by the above discussion and the proof of Theorem~\ref{GVD_main}, we start by taking the initial term of $-f^{w_0}_{5,3}$ which is $y_2 :=x_{1,4}$ and we define the ideals \[N_2 := I_{w_0,h}(3) :=  \langle f^{w_0}_{3,1},f^{w_0}_{4,1},f^{w_0}_{4,2}\rangle \, \textup{ and } \]
\[ C_2 := \langle 1 \rangle.\]

\noindent Since $C_2=\langle 1 \rangle = \CC[\mathbf{x}_{w_0}]$ and hence $I_{w_0,h}(2) \cap C_2 = I_{w_0,h}(2)$, it immediately follows that $N_2\subset I_{w_0,h}(2)\cap C_2$.  Moreover, $I_{w_0,h}(2)$ is square-free in $y_2$, and no term of the reduced Gr\"obner basis for $N_2$ with respect to $<_5$ is divisible by $y_2$. Next, observe that \[I_{w_0,h}(2)/N_2 \cong f^{w_0}_{5,3}R/N_2 \] 
since the only generator of $I_{w_0,h}(2)$ not contained in $N_2$ is $f^{w_0}_{5,3}$. It is also free (of rank $1$) as an $R/N_2$-module since there is a term in $f^{w_0}_{5,3}$ which contains the variable $y_2$, and thus (the equivalence class of) $f^{w_0}_{5,3}$ is not a zero-divisor in $R/N_2$. Since
$C_2=R$, we clearly have $C_2/N_2 = R/N_2$, so there exists an isomorphism $C_2/N_2\rightarrow I_{w_0,h}(2)/N_2 \cong f^{w_0}_{5,3}R/N_2$ given by multiplication by $-f^{w_0}_{5,3}$. Since $\init_{<_5}(-f^{w_0}_{5,3}) = y_2$, we can now apply Theorem  \ref{theorem: Klein Rajchgot} to conclude that \[\init_{y_2}(I_{w_0,h}(2)) = C_2\cap(N_2 +\langle y_2 \rangle)\] defines a geometric vertex decomposition of $I_{w_0,h}(2)$.

To see that $I_{w_0,h}(2)$ is geometrically vertex decomposable, we must show next that the contractions of $C_2$ and $N_2$ to $\CC[\mathbf{x}_{w_0} \setminus y_2] = \CC[\mathbf{x}_{w_0} \setminus \{x_{1,4}\}]$ are both geometrically vertex decomposable. Since $C_2=\langle 1 \rangle = R$, it follows that its contraction is also the unit ideal, so it is geometrically vertex decomposable. We now observe that, upon changing variable labels as explained in Example~\ref{injection_example}, $N_2=I_{w_0,h}(3)$ can be interpreted as the Hessenberg patch ideal of the regular nilpotent Hessenberg variety with  $\bar{h}=(2,3,4,4)$ in the $\overline{w_0}$-chart where $n=4$. Interpreted in this way, $N_2$ is geometrically vertex decomposable by induction on $n$. Applying Theorem~\ref{theorem: extension GVDs} to $I=N_2$ (interpreted in $\CC[\mathbf{x}_{\overline{w_0}}]$) and $J=0$ as in the proof of Theorem~\ref{GVD_main}, we see that $N_2$ (interpreted in $\CC[\mathbf{x}_{w_0}])$ is geometrically vertex decomposable. From this we may conclude that $I_{w_0,h}(2)$ is geometrically vertex decomposable. 

We may now repeat this process as follows to show that $I_{w_0,h}(1)$ is geometrically vertex decomposable. Namely, we may choose $y_1 = x_{1,3}$ and $N_1=I_{w_0,h}(2)$ and $C_1=\langle 1\rangle$. By analogous arguments, $\init_{y_1}(I_{w_0,h}(1)) = C_1 \cap (N_1 + \langle y_1 \rangle)$, and $C_1$ (being the unit ideal) is geometrically vertex decomposable. Also, we just showed $N_1$ is geometrically vertex decomposable. 

Finally, following the same procedure, we can see that $I_{w_0,h} = I_{w_0,h}(0)$ is geometrically vertex decomposable by taking $y_0=x_{1,2}$, $N_0 = I_{w_0,h}(1)$ and $C_0=\langle 1\rangle$.
\end{example}

\bigskip

We have just seen that the ideals $I_{w_0,h}$ are GVD and gone through a specific example. Lemma~\ref{lemma: GVD to GB} now allows us to conclude that we have a Gr\"obner basis. 

\begin{theorem}\label{Hess_Grobner}
Let $n \geq 3$ be a positive integer, and let $h: [n] \to [n]$ be an indecomposable Hessenberg function. 
Then the set of elements $f^{w_0}_{k,\ell}$, which generate the Hessenberg patch ideal $I_{w_0,h}$ of $\hess(\mathsf{N},h)$ in the $w_0$-chart, form a Gr\"obner basis for $I_{w_0,h}$ with respect to the monomial order $<_n$. Moreover,  $\init_{<_n}(I_{w_0,h})$ is the ideal of indeterminates given by  \[\init_{<_n}(I_{w_0,h}) = \langle x_{n-i+1,j+1} |(i,j)\in H(h,0)\rangle. \]
\end{theorem}

\begin{proof}
In the proof of Theorem~\ref{GVD_main} we have checked all the hypotheses of Lemma \ref{lemma: GVD to GB}. Hence by Theorem~\ref{GVD_main} and Lemma~\ref{lemma: GVD to GB}, we may conclude that $f^{w_0}_{k,\ell}\in I_{w_0,h}$ define a Gr\"obner basis with respect to $<_n$.  The second statement follows from the fact that $\init_{<_n}(f^{w_0}_{k,\ell}) = - x_{n-k+1,\ell+1}$.
\end{proof}

\end{document}